\documentclass[11pt,a4paper]{scrartcl}
\usepackage[utf8]{inputenc}
\usepackage[english]{babel}

\usepackage{enumerate}
\usepackage{latexsym} 
\usepackage{amsmath}
\usepackage{amsthm}
\usepackage{amssymb} 
\usepackage{etoolbox} 

\usepackage{tikz}
\usepackage{tikz-cd}
\usetikzlibrary{cd, matrix,arrows,decorations.pathmorphing, babel}
\usetikzlibrary{decorations.markings}
\usetikzlibrary{decorations.pathmorphing}
\usepackage{xcolor}
\usepackage{pgfplots}
\usepgfplotslibrary{polar}

\usepackage{thmtools}
\usepackage{thm-restate}

\usepackage[obeyspaces]{url} 
\usepackage{hyperref}
\usepackage{cleveref}
\usepackage[all]{hypcap} 

\newtheoremstyle{defistyle}
{}
{}
{}
{}
{\bfseries}
{}
{\newline}
{}

\newtheoremstyle{bewstyle}
{}
{}
{}
{}
{\bfseries}
{}
{3 pt}
{}

\theoremstyle{defistyle}
\newtheorem{defi}{Definition}[section]
\newtheorem{Bem}[defi]{Remark}
\newtheorem{Not}[defi]{Notation}

\newtheorem{Satz}[defi]{Theorem}
\newtheorem{Prop}[defi]{Proposition}
\newtheorem{theorem}[defi]{Theorem}
\newtheorem{Lem}[defi]{Lemma}
\newtheorem{Bsp}[defi]{Example}
\newtheorem{Kor}[defi]{Corollary}

\theoremstyle{bewstyle}
\newtheorem*{Bew}{proof:}
\AtEndEnvironment{Bew}{\qed}

\newcommand{\N}{\mathbb{N}}
\newcommand{\R}{\mathbb{R}}

\newcommand{\Z}{\mathbb{Z}}
\newcommand{\danger}{{\fontencoding{U}\fontfamily{futs}\selectfont\char 66\relax} \ }

\DeclareMathOperator{\Env}{Env}
\DeclareMathOperator{\EnvR}{Env_R}
\DeclareMathOperator{\cand}{cand}
\DeclareMathOperator{\vol}{vol}

\DeclareMathOperator{\supp}{supp}
\DeclareMathOperator{\relint}{relint}

\DeclareMathOperator{\Aut}{Aut}
\DeclareMathOperator{\Out}{Out}
\DeclareMathOperator{\Inn}{Inn}
\DeclareMathOperator{\Isom}{Isom}

\DeclareMathOperator{\PGL}{PGL}

\usepackage{todonotes}

\begin{document}

\title{Envelopes in Outer Space}
\author{Christian Steinhart}

\maketitle

\begin{abstract}
We study the geometry of Outer Space $CV_n$ in regard of the asymmetric Lipschitz metric via envelopes, that is the set of all geodesics between two points. In the simplicial structure of $CV_n$ the envelopes are polytopes. We construct a piecewise unique geodesic between any two points in $CV_n$ by concatenating edges of these polytopes. In fact rigid geodesics can be identified with edges of out- and in-envelopes, that is the set of all geodesics from or to a base point with a given maximally stretched path.

We introduce a notion of general position for pairs of points which is a dense and open condition. Using this we will show, that for almost all pairs of points in $CV_n$ their envelopes have dimension $3n-4$.

Whenever an envelope passes a face, it might change its dimension. This determines the simplicial structure of reduced Outer Space via the Lipschitz metric which implies 
$\Isom(CV_n^{red})=\Isom(CV_n)$. As another implication we get that a geodesic ray in $CV_2$ becomes after a given length rigid.
\end{abstract}

\tableofcontents

\section*{Introduction}
The Outer Space $CV_n$ also denoted $X_n$ was first introduced in 1986 by Culler and Vogtmann in \cite{CV86} to study the outer morphism group $Out(F_n)$.
They showed that $CV_n$ is contractible and $Out(F_n)$ acts with finite point stabilizers.
$CV_n$ can be seen as the Teichmüller analogon for metric graphs were $Out(F_n)$ plays the role of the mapping class group. So $CV_n$ is the moduli space of finite, marked, metric graphs with no leaves and fundamental group $F_n$. Reduced Outer Space $CV_n^{red}$ is the subset of $CV_n$ consisting of all elements with no separating edges in the corresponding graph. It is a deformation retract of $CV_n$.

 $CV_n$ is a $3n-4$ dimensional simplicial complex with some missing faces, where each open simplex corresponds to a graph with a marking. There are some missing faces, since we can't contract loops without changing the fundamental group of a graph.

Most of the previous work and results in Outer Space were done topologically and combinatorially. In 2008 Francaviglia and Martino introduced in \cite{FM11} a natural asymmetric metric on $CV_n$ similar to the Thurston metric in Teichmüller space (s. \cite{T86}) called the Lipschitz metric. Like the Thurston metric it can be calculated as the supremal stretching of curves. Stefano and Martino showed that for each $A \in CV_n$ there always exists a finite set of curves called candidates such that at least one of them is maximally stretched from $A$ to a $B \in CV_n$. Hence this distance can easily be calculated for any two points for example with the package \cite{CVcalc} written in Sage (\cite{sagemath}).

Some interesting properties of the Lipschitz metric are, that geodesics are almost never unique and that a geodesic from $A$ to $B$ does not have to be a geodesic from $B$ to $A$. In fact $CV_n$ is a geodesic space for the asymmetric metric, but not every pair can be connected with a symmetric geodesic. The freedom of a geodesic between two points can be described in terms of envelopes, which is the set of all points which lie on a geodesic from $A$ to $B$. In the case of Teichmüller space with the Thurston metric envelopes have been studied for the punctured torus in \cite{DLRT}. The envelopes in Outer Space can also be used to describe a coarse sense of direction, namely given a geodesic segment the out- and in-going envelopes describe all geodesics starting resp. ending with that segment.

Unique geodesics, which are also called rigid, play a crucial role in the understanding of the geometry of Outer Space, for example in the proof in \cite{FM12} that the isometry group $\Isom(CV_n)$ is $Out(F_n)$ for $n \geq 3$. In chapter 3 we will use envelopes to give an algorithm which yields a piecewise unique geodesic between any two points by concatenating consecutive edges of an envelope. This yields the following theorem:

\newtheorem*{thm:pwRigid}{Theorem \ref{unique geodesic segments}}
\begin{thm:pwRigid}
For each $A,B \in CV_n$ there exist geodesic segments $g_1, \dots, g_l$, s.t. for each $i$ $g_i$ is the unique asymmetric geodesic joining its two endpoints and $g=g_1 * g_2 * \dots *g_l$ is an asymmetric geodesic from $A$ to $B$.
\end{thm:pwRigid}

In fact theorem \ref{theorem:unique_geodesics_are_edges} states that rigid geodesic are exactly the edges of out- and in-envelopes.

It is an interesting fact that the dimension of an envelope may decrease whenever it passes a face. We can use this to determine faces in the reduced Outer Space in terms of envelopes and since envelopes are preserved under isometry we get that isometries of reduced Outer Space are simplicial. This was the missing step in \cite{FM12} for the reduced case and hence we get theorem

\newtheorem*{thm:isomRedCV}{Theorem \ref{THEOREM2}}
\begin{thm:isomRedCV}
The isometry groups of $CV_n^{red}$ in regard of the symmetric and both asymmetric Lipschitz-metrics are the same as in the non-reduced case: \[\Isom(CV_n^{red})=\Isom(CV_n)=\begin{cases}
\Out(F_n)&, \text{ if } n \geq 3\\
\PGL(2,\Z)&, \text{ if } n=2
\end{cases}\]
\end{thm:isomRedCV}

The decrease of dimension can be explained in terms of maximally stretched candidates. If the originally maximally stretched candidate $\gamma$ is no longer a candidate in an adjacent simplex, there must be at least two other candidates of that new face which are maximally stretched along $\gamma$. This gives a restriction on the points lying in the envelope and we may loose a dimension of freedom for all geodesics. In particular if a geodesic ray $g$ runs long enough, it must cross such a face and we get that each pair of points far enough in $g$ lie in a special position to each other. As a nice corollary we get, that long enough geodesic rays in $CV_2$ become at some point rigid. 

\subsection*{Roadmap}
\textbf{Section 1} will give the basic definitions of Outer Space and the Lipschitz metric. Although this section is enough to understand the rest of the paper I strongly recommend Vogtmanns surveys \cite{V15} or \cite{V02} to get familiar with this topic.

\textbf{Section 2} will give the notions and main properties of geodesics in Outer Space. Furthermore it contains some basic tools which we will use in the following chapters.

\textbf{Section 3} introduces envelopes and shows that they are polytopes in $CV_n$. Theorem \ref{unique geodesic segments} then shows how one can walk along edges of those polytopes to get a piecewise rigid geodesic. We will also see, that out- and in-going envelopes cover $CV_n$ and in each simplex the possible out-/in envelopes are determined by the candidate envelopes.

In \textbf{Section 4} we will first show, that an envelope has almost always maximal dimension. On the other hand in the reduced case we will always find an envelope close to each point in a face, such that the envelope decreases its dimension when it passes the face. Using this we get that isometries respect the simplicial structure of reduced Outer Space.

\section*{Acknowledgements}
I thank Armando Martino for three great weeks in Southampton, where the Appendix of this paper and its corresponding results were created and for great input and conversations, which deepened my understanding of the theory.\\
I heartly thank Gabriela Weitze-Schmithüsen for her thorough proof-reading, discussions and general help which greatly served the readability of the paper.\\
I also thankfully acknowledge the support by DFG-collaborative research center TRR 195 (Project I.8).

\section{Preliminaries}
This section will give a quick introduction into the basic definitions and properties of Culler-Vogtmann Outer Space and the Lipschitz metric on it. For a more thorough introduction and survey I refer the reader to [...].

The definition of Outer Space is analogue to the definition of Teichmüller space just in terms of graphs. More explicitly a point in Outer Space $CV_n$ consists of three data, namely a finite graph $\Gamma$ without leaves, a marking on $\Gamma$ and lengths of its edges:

\begin{defi}
\begin{enumerate}[(i)]
\item For $n \in \N$ the \emph{rose} $R_n$ is the graph with one vertex and $n$ edges, also called petals. Hence we can easily identify the fundamental group $\pi_1(R_n)$ with the free group $F_n$ by assigning each (oriented) petal a basis element of the free group.
\begin{figure}[h]
\center
\begin{tikzpicture}[scale=0.6, decoration={
    markings,
    mark=at position 0.5 with {\arrow{>}}}]
\def\n{9}
\begin{polaraxis}[grid=none, axis lines=none, clip=false]
\pgfplotsinvokeforeach{1,...,9}{
     \addplot[mark=none,domain=(#1-0.5)*(360/\n):(#1+0.5)*(360/\n),samples=300/\n,gray,
      postaction=decorate] { abs(cos(\n*x/2))};
      \node[gray] at (40+360/9*#1:110){$x_#1$};
}
\filldraw[gray] (0,0) circle[radius=1pt];
\end{polaraxis}
\end{tikzpicture}
\caption{The rose graph $R_n$ with labeled petals}
\end{figure}
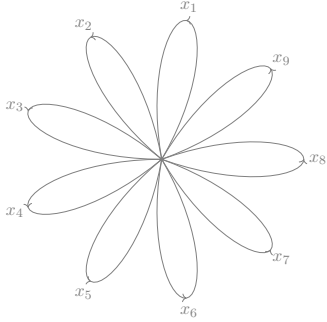

\item Let $\Gamma$ be a finite graph, where each vertex has at least valency 3. A \emph{marking on $\Gamma$} is a homotopy equivalence $m: R_n \to \Gamma$. Hence we have now identified $F_n$ with the fundamental group $\pi_1(\Gamma)$. 
\item The \emph{(projectivized) Outer Space of rank $n$} is defined as the set
\begin{align*}
CV_n:=\{ (\Gamma, l, m) \ | \ &\Gamma \text{ is a finite graph with all vertices have at least valency 3}\\
& \text{ with edge lengths } l: E(\Gamma) \to \R_{>0} \text{ , i.e. } (\Gamma, l) \text{ is a metric graph}\\
& m: R_n \to \Gamma \text{ is a marking} \} /\sim
\end{align*}
Where we have the following equivalence relation: $(\Gamma, l, m) \sim  (\Gamma', l', m')$ if there exists a homothety $h: (\Gamma, l) \to (\Gamma', l')$ such that the induced marking $h \circ m$ is homotopic to $m'$
\item The \emph{reduced Outer Space $CV_n^{red}$} is the same set as above, but with the restriction that our graphs have no separating edges. This is a deformation retract of $CV_n$.
\end{enumerate}
\end{defi}

In other words our points in $CV_n$ can be considered as finite metric graphs without leaves and with volume $vol(\Gamma)=\sum\limits_{e \in E(\Gamma)} l(e) =1$ (by homothety we can stretch the graph) where we identify the free group $F_n$ with the fundamental group of $\Gamma$ via the marking up to free homotopy.

\begin{Not}
\begin{enumerate}[(i)]
\item From now on we will omit the length and marking in our notation and mean them implicit in $\Gamma$.
\item Quite often it is easier to write down a homotopy inverse $\tilde{h}$ of the marking as follow: Fix a spanning tree of $\Gamma$ which will be collapsed to the vertex of $R_n$ and label the rest of the edges with an orientation and a basis of $F_n$. Each edge will be sent to the sequence of petals in $R_n$ corresponding to its label (s. figure \ref{spaceship}).

\begin{figure}[h]\label{spaceship}
\center
\begin{tikzpicture}[scale=0.7, decoration={
    markings,
    mark=at position 0.5 with {\arrow{>}}}
    ] 

\draw[red,postaction=decorate] (0,2) to[out=180, in=180] node[left]{$x_2$} (0,1);
\draw[blue,postaction=decorate] (0,2) --node[above]{$x_1$} (2,2);
\draw[green,postaction=decorate] (4,2.5)--node[below]{$x_5$}(6,1.8);

\begin{scope}[yscale=-1,shift={(0,-1)}]
\draw[orange,postaction=decorate] (0,2) to[out=180, in=180] node[left]{$x_2x_3$} (0,1);
\draw[cyan,postaction=decorate] (0,2) -- node[below]{$x_1^{-1}x_4x_2$}(2,2);
\draw[purple,postaction=decorate] (4,2.5)--node[above]{$x_9$}(6,1.8);
\end{scope}

\draw[brown,postaction=decorate] (2,1) -- node[right]{$x_6$} (2,0);
\draw[pink, postaction=decorate] (6,1.8) -- node[right]{$x_6x_7$}(6,-0.8);
\draw[yellow, postaction=decorate] (4,2.5) -- node[right]{$x_6x_8$}(4,-1.5);

\draw (0,2) node{$\bullet$} -- (0,1) node{$\bullet$} -- (2,1) node{$\bullet$} -- (2,2) node{$\bullet$} -- (2,2.5)  to[out=90, in=-210] (4,2.5) node{$\bullet$} -- (4,3.2) to[out=90, in=-220] (6,1.8) node{$\bullet$} -- (6.3,1.6) to[out=-40, in=90] (7.5,0.5);

\begin{scope}[yscale=-1,shift={(0,-1)}]
\draw (0,2) node{$\bullet$} -- (0,1) node{$\bullet$} -- (2,1) node{$\bullet$} -- (2,2) node{$\bullet$} -- (2,2.5)  to[out=90, in=-210] (4,2.5) node{$\bullet$} -- (4,3.2) to[out=90, in=-220] (6,1.8) node{$\bullet$} -- (6.3,1.6) to[out=-40, in=90] (7.5,0.5);
\end{scope}

\begin{scope}[shift={(10,-3.3)}]
\def\n{9}
\begin{polaraxis}[grid=none, axis lines=none, clip=false]
\pgfplotsinvokeforeach{1,...,9}{
     \addplot[mark=none,domain=(#1-0.5)*(360/\n):(#1+0.5)*(360/\n),samples=300/\n,gray,
      postaction=decorate] { abs(cos(\n*x/2))};
      \node[gray] at (40+360/9*#1:120){$x_#1$};
}

\addplot[blue,mark=none,domain=65:95,samples=50, postaction=decorate] { abs(cos(\n*((65-x)*4/3+60)/2))*0.9};
\addplot[red,mark=none,domain=105:135,samples=50, postaction=decorate] { abs(cos(\n*((105-x)*4/3+100)/2))*0.9};
\addplot[orange,mark=none,domain=100:140,samples=50, postaction=decorate] { abs(cos(\n*x/2))+(x-100)/40*0.2};
\addplot[orange,mark=none,domain=140:180,samples=50, postaction=decorate] { abs(cos(\n*x/2))+(180-x)/40*0.2};

\addplot[cyan,mark=none,domain=95:65,samples=50, postaction=decorate] { abs(cos(\n*((65-x)*4/3+60)/2))*0.95};
\addplot[cyan,mark=none,domain=185:215,samples=50, postaction=decorate] { abs(cos(\n*((185-x)*4/3+180)/2))*0.9};
\addplot[cyan,mark=none,domain=105:135,samples=50, postaction=decorate] { abs(cos(\n*((105-x)*4/3+100)/2))*0.95};

\addplot[green,mark=none,domain=225:255,samples=50, postaction=decorate] { abs(cos(\n*((225-x)*4/3+220)/2))*0.9};
\addplot[brown,mark=none,domain=265:295,samples=50, postaction=decorate] { abs(cos(\n*((265-x)*4/3+260)/2))*0.9};

\addplot[pink,mark=none,domain=260:300,samples=50, postaction=decorate] { abs(cos(\n*x/2))+(x-260)/40*0.2};
\addplot[pink,mark=none,domain=300:340,samples=50, postaction=decorate] { abs(cos(\n*x/2))+(340-x)/40*0.2};

\addplot[yellow,mark=none,domain=265:295,samples=50, postaction=decorate] { abs(cos(\n*((265-x)*4/3+260)/2))*0.95};
\addplot[yellow,mark=none,domain=345:375,samples=50, postaction=decorate] { abs(cos(\n*((345-x)*4/3+340)/2))*0.9};

\addplot[purple,mark=none,domain=25:55,samples=50, postaction=decorate] { abs(cos(\n*((25-x)*4/3+20)/2))*0.9};

\filldraw (0,0) circle[radius=2pt];
\end{polaraxis}
\end{scope}
\end{tikzpicture}
\caption{A typical element in Outer Space and its homotopy inverse image in $R_n$}

\end{figure}
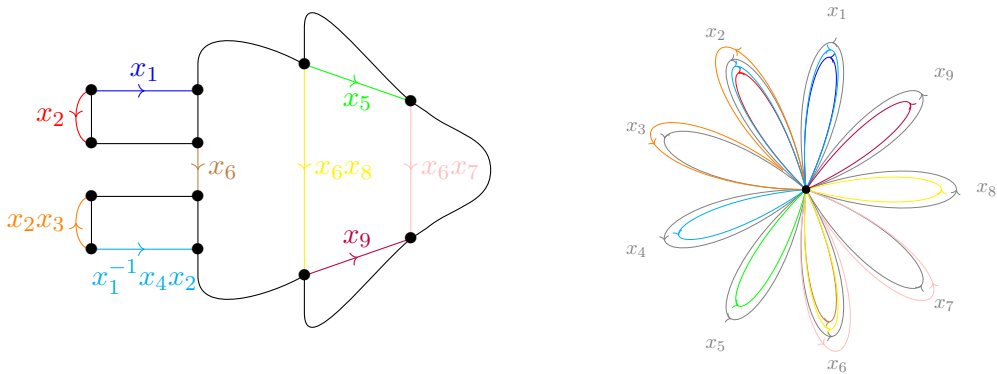

\item From now on we will identify each element of $\pi_1(\Gamma)$ with its image under the marking and vice versa.
\item Since every (conjugacy class of an) element of $\pi_1(\Gamma)$ can be uniquely realized as an immersed loop in $\Gamma$, we can assign for each triple $(\Gamma, l, m)$ and element $\gamma \in F_n$ a length of the corresponding immersed realisation in $\Gamma$. We will also denote this by $l_{\Gamma}(\gamma)$.
\end{enumerate}
\end{Not}

The data $(\Gamma,m)$ of an element $A=(\Gamma, l,m) \in CV_n$ is called the topological type of $A$. After normalizing the volume we can see, that all the points in $CV_n$ with the same topological type build an open simplex.\\
If we can pass from one marked graph $(\Gamma,m)$ to another marked graph $(\Gamma',m')$ by collapsing a forest, then we identify the open simplex of $(\Gamma',m')$ with the missing face of the simplex of $(\Gamma,m)$ were the edges of the forest have length 0.

Hence we can represent $CV_n$ as a simplicial complex with some missing faces and consider the induced topology.

Culler and Vogtmann showed in \cite{CV86} the following two important theorems:

\begin{theorem}[Culler Vogtmann 86]
$CV_n$ is contractible.
\end{theorem}

On $CV_n$ we have a natural right action of $\Aut(F_n)$ by change of marking, i.e. each element $\phi$ can be realised as a homotopy equivalence $\phi' : R_n \to R_n$ and precomposing to the marking yields the action on $CV_n$: $(\Gamma, l, m) \bullet \phi := (\Gamma, l, m \circ \phi')$. Furthermore inner automorphisms act trivially on $CV_n$ since the marking is only defined up to homotopy, hence we actually have an action of $\Out(F_n):=\Aut(F_n)/\Inn(F_n)$.

\begin{theorem}[Culler Vogtmann 86]
The $\Out(F_n)$-action on $CV_n$ is fix-point free and each point has a finite stabilizer.
\end{theorem}  

Similarly to the Thurston metric on Teichmüller space as introduced in \cite{T86} Armando Martino and Stefano Francaviglia introduced in \cite{FM11} an asymmetric metric on $CV_n$ as following:

\begin{defi}\label{defi:Lipschitz-distance}
Let $A:=(\Gamma, l, m),B:=(\Gamma', l', m') \in CV_n$. Then consider the set $S$ of all continuous maps $h: (\Gamma,l) \to (\Gamma',l')$, s.t. $m \circ h \cong m'$, i.e. the following diagramm commutes up to homotopy:
\[
\begin{tikzcd}
R_n \arrow[r, "m"] \arrow[rd, "m \prime "] & \Gamma \arrow[d, "h"]\\
& \Gamma \prime
\end{tikzcd}
\]

Since finite metric graphs are compact, each $h \in S$ is Lipschitz continuous with Lipschitz-constant $L(h)$. The \emph{Lipschitz distance} from $A$ to $B$ is then defined as
\begin{align*}
d_R(A,B)&:= \log ( \frac{\vol(\Gamma)}{\vol(\Gamma')} \cdot \Lambda_R(A,B) )\\
\text{ with } \quad \Lambda_R(A,B)&:=\inf_{h \in S} L(h)
\end{align*}
The typical way to gain now the symmetric Lipschitz distance is by:
\[d(A,B):=d_R(A,B)+d_R(B,A)\]
\end{defi}

The Arzela-Ascoli theorem yields, that the infimum $\Lambda_R(A,B)$ is actually attained by a map $h \in S$. There is even an easier and more explicit way to calculate this distance by looking at the maximal stretching of certain paths in the graph.

\begin{defi}
For a given graph $\Gamma$ a \emph{candidate} is a simple loop in $\Gamma$ whose image is a topological embeddeding of one of the following graphs (s. figure \ref{fig:candidates}):
\begin{itemize}
\item a simple loop
\item a figure of eight
\item or a barbell
\end{itemize}
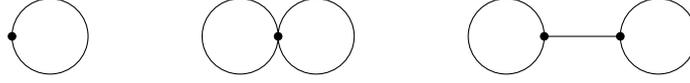
\begin{figure}[h]\label{fig:candidates}
\centering
\begin{tikzpicture}
\draw (0,0) arc (180:-180:1/2);
\draw[fill] circle [radius=0.05];
\begin{scope}[shift={(3.5,0)}]
\draw (0,0) arc (180:-180:1/2);
\draw (0,0) arc (0:360:1/2);
\draw[fill] circle [radius=0.05];
\begin{scope}[shift={(3.5,0)}]
\draw (0,0) arc (0:360:1/2);
\draw[fill] circle [radius=0.05];
\draw (0,0) -- (1,0);
\draw (1,0) arc (180:-180:1/2);
\draw[fill] (1,0) circle [radius=0.05];
\end{scope}
\end{scope}
\end{tikzpicture}
\caption{A simple loop, a figure of eight and a barbell}
\end{figure}

We will denote the set of candidates of $\Gamma$ with $\cand(\Gamma)$ and identify them via the marking as the corresponding subset (of conjugacy classes) in $F_n$.
\end{defi}

Now a theorem by Francaviglia and Martino (s. \cite{FM11}) states, that the minimal Lipschitz constant can be calculated as the maximal stretching of these candidates:

\begin{theorem}\label{1 candidate is enough}
For $A, B \in CV_n$ we have:
\[\Lambda_R(A,B) = \sup_{\alpha \in F_n} \frac{l_B(\alpha)}{l_A(\alpha)} = \max_{\alpha \in \cand(A)} \frac{l_B(\alpha)}{l_A(\alpha)}\]
\end{theorem}

Since there are only finitely many candidates in a graph, we can now easily compute $\Lambda_R(A,B)$. A computational realization of this theorem with Sage (\cite{sagemath}) can be found under \cite{CVcalc}.

\begin{defi}
We say $\gamma \in F_n$ is a \emph{witness} for $A$ to $B$ for the asymmetric metric, if $\gamma$ is maximally stretched from $A$ to $B$, i.e. $\Lambda_R(A,B)=\frac{l_B(\gamma)}{l_A(\gamma)}$. We denote the set of witnesses from $A$ to $B$ as $W_R(A,B)$.

We say $(\gamma, \omega) \in F_n\times F_n$ is a \emph{(symmetric) witness} for $A$ and $B$ for the symmetric metric, if $\gamma$ is an asymmetric witness from $A$ to $B$ and $\omega$ is an asymmetric witness from $B$ to $A$. We denote the set of symmetric witnesses for $A$ and $B$ as $W(A,B)$.
\end{defi}

We may think of a witness as the conjugacy class of a simple element, as taking conjugates or powers/roots doesn't change the metric behaviour. Most of the time the witnesses we consider will be candidates of $A$. Such witnesses will be called \emph{candidate witnesses}. We will denote the set of candidate witnesses by $CW(A,B):=\cand(A)\cap W_R(A,B)$. We will see, that they play a crucial role in the study of geodesics.

\section{Geodesics in Outer Space}
A first step to understand a metric space is often to understand its (minimizing) geodesics. For example they played the crucial role in the understanding of the isometry group of Outer Space in \cite{FM12}. The philosophy of this paper is, that a lot of information, e.g. the simplicial structure of $CV_n$, is encoded in the geodesics. In this section some basic properties of the geodesics in Outer Space are listed.

As in (symmetric) metric spaces we define length and geodesics in asymmetric metric spaces:
\begin{defi} \label{2 def geodesic}
Let $(X,d)$ be a space with an (asymmetric) metric $d$, $I \subseteq \R$ an interval and $g: I \to X$ a continuous path.
\begin{enumerate}[(i)]
\item Let $J:=[a,b] \subseteq I$ be a closed interval. Then the length of the arc $g|_J$ is defined as
\[l(g|_J):= \sup \{\sum_{i=1}^N d(g(t_{i-1}),g(t_i))   \ | \ N \in \N, a=:t_0 \leq t_1 \leq \dots \leq t_N:=b\}.\]
We call $g$ \emph{rectifiable} iff the arc length is finite for every arc of $g$.
\item A rectifiable curve $g$ is called a \emph{minimizing geodesic} iff for every arc $J:=[a,b] \subseteq I$ we have $l(g|_J)=d(g(a),g(b))$.
\end{enumerate}
\end{defi}

\danger Keep in mind, that these definitions depend on the orientation of $g$, since we can have $d(g(t_1),g(t_2)) \neq d(g(t_2),g(t_1))$. Especially if $g$ is a $d_R$-geodesic we can still have that $\overline{g}(t):=g(-t)$ is not a $d_R$-geodesic!

Another way to define minimizing geodesics is as follows

\begin{Lem}\label{2 lemma triangle ineq and geodesics}
Let $(X,d)$ and $g$ be as in definition \ref{2 def geodesic}.\\
Then we have that $g$ is a minimizing geodesic if and only if it realises the triangle equality
\[d(g(a),g(c))=d(g(a), g(b)) + d(g(b), g(c))\]
for all $a,b,c \in I$ with $a \leq b \leq c$.
\end{Lem}

\begin{Bew}
Let $a \leq b \leq c \in I$
\glqq $\Rightarrow$ \grqq: By the triangle inequality of the metric we have
\begin{align*}
d(g(a),g(c)) &\leq d(g(a), g(b)) + d(g(b), g(c))\\
&\leq l(g|_{[a,c]})=d(g(a),g(c))
\end{align*}
where the last inequality is the definition of length and the last equality is the definition of minimizing geodesic. Hence equality holds.

\glqq $\Leftarrow$ \grqq: Let $a:=t_0 \leq \dots \leq t_N:=b$ be any subdivision of $[a,b]$, then iteratively applying the triangle equality yields
\[\sum_{i=1}^N d(g(t_{i-1}),g(t_i))= d(g(t_0),g(t_2))+\sum_{i=3}^N d(g(t_{i-1}),g(t_i))= \dots = d(g(t_0),g(t_N))\]
and hence $l(g|_{[a,b]})=d(g(a),g(b))$
\end{Bew}

\begin{Bem}
In other metric spaces sometimes the definition of a geodesics differ, namely one additionally requires a geodesic $g$ to be parametrized by length, i.e. $l(g|_{[a,b]})=b-a$ and relaxes the conditions of definition \ref{2 def geodesic} (ii) to be satisfied only locally. Lemma \ref{glueing geodesics} will tell us, how minimizing geodesics in Outer Space are glued together to get such geodesics (i.e. locally minimizing geodesics). Hence we will restrict to study minimizing geodesics. Another good reason to stick to minimizing geodesics instead of locally minimizing geodesics
is that being a locally minimizing geodesic in Outer Space is a relatively weak condition, as you can see for example in remark \ref{2 fun stuff with locally-minimizing geodesics}. 
\end{Bem}

\begin{Not}
From now on we denote by \emph{geodesic} a minimizing geodesic in regards of the asymmetric metric $d_R$ and by \emph{symmetric geodesic} a minimizing geodesic in regards of the symmetric metric $d$ in $CV_n$.
\end{Not}

The notion of symmetric geodesic is not only because it is a geodesic in the symmetric metric, but also because these are exactly those asymmetric geodesics which are still geodesic if you flip the direction:

\begin{Bem}
A continuous path $g: I \to CV_n$ is a symmetric geodesic if and only if $g$ and $\overline{g}$ are asymmetric geodesics, where $\overline{g}(t):=g(-t)$ denotes $g$ with a flipped orientation. This can easily be seen by applying the previous lemma \ref{2 lemma triangle ineq and geodesics}.
\end{Bem}

It is always an important question about a metric space, if there exists a geodesic joining two points, i.e. if the metric space is geodesic. In \cite[theorem 5.5]{FM11} Francaviglia and Martino proved by \glqq folding the edges\grqq \ according to an optimal Lipschitz map the following theorem

\begin{theorem}
For any two points $A,B \in CV_n$ there exists a geodesic from $A$ to $B$.
\end{theorem}

Also mentioned in the same paper we have a few remarks:

\begin{Bem}\label{c2_geodesics_in_CVn}
\begin{enumerate}[(i)]
\item There does not always exist a symmetric geodesic between two points (s. the following example).
\item On the other hand straight lines in the simplices are symmetric geodesics.
\item Geodesics are almost never unique, in fact for each point $A \in CV_n$ there are only finitely many rigid geodesics emanating from $A$. Here we call a geodesic rigid, if any sub-arc is the unique geodesic joining its endpoints. Keep in mind that a geodesic between two points is rigid if and only if it is the unique geodesic joining its endpoints.
\end{enumerate}
\end{Bem}

Together with Armando Martino we worked out the following extreme example for the last remark:

\begin{Bsp}\label{non_geodesic_ball_dim2}
Let $X \in CV_2^{red}$ be a point in a 1-dimensional simplex, i.e. a figure of eight graph, and $U \subseteq CV_2^{red}$ a neighbourhood of $X$. Then there exist points $A,C \in U$ such that there is no symmetric geodesic between $A$ and $C$.
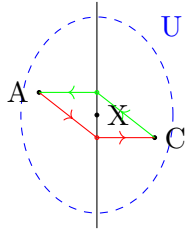
\begin{figure}[h]
\centering
\begin{tikzpicture}
\draw (0,-1.5) -- (0,1.5);
\draw[fill] (0,0) circle[radius=0.025] node[right]{X};
\begin{scope}[yscale=1.3]
\draw[dashed, blue] (0,0) circle[radius=1];
\node[blue] at (0.7, 0.7)[above right] {U};
\end{scope}
\draw[fill] (-0.76,0.3) circle[radius=.025] node[left]{A};
\draw[fill] (0.76,-0.3) circle[radius=0.025] node[right]{C};
\begin{scope}[decoration={
    markings,
    mark=at position 0.5 with {\arrow{>}}}
    ] 
\draw[postaction=decorate, red] (-0.76,0.3) -- (0,-0.3) circle[radius=0.025];
\draw[postaction=decorate, red] (0,-0.3) -- (0.76,-0.3);
\draw[postaction=decorate, green] (0.76,-0.3) -- (0,0.3) circle[radius=0.025];
\draw[postaction=decorate, green] (0,0.3) -- (-0.76,0.3);
\end{scope}
\end{tikzpicture}
\caption{$A,C,X$ and $U$ in $CV_2^{red}$ with asymmetric geodesics}
\end{figure}
The reason this works is the following idea:

Each geodesic from $A$ to $C$ must intersect the face of $X$ at at least one point $B$ and each geodesic from $C$ to $A$ must pass through the face of $X$ at some point $B'$.
If there exists a symmetric geodesic, then $B$ and $B'$ must coincide. But we can choose $A$ and $C$ in such a manner, that all possible intersection points $B$ and $B'$ are disjoint.

You can find the exact calculations for this example in the appendix as remark \ref{bem:app_calc_non_geodesic}. Interesting is here the fact, that this highly depends on the topological type of $A$ and $C$. For example if we choose $A$ or $C$ to be a barbell graph, then this doesn't work since each barbell graph has a symmetric geodesic to all adjacent theta-graphs. You can find this statement also in the appendix as remark \ref{bem:app_geod_join_barb_theta}. The proof will use lemma \ref{keep witness}.
\end{Bsp}

Using the fact that geodesics are sent to geodesics under isometries remark \ref{c2_geodesics_in_CVn} (ii) and example \ref{non_geodesic_ball_dim2} imply the following corollary:

\begin{Kor}
Let $h \in \Isom(CV_2^{red})$, then $h$ is simplicial, i.e. $h$ maps simplices to simplices and faces to faces.
\end{Kor}

The non-triviality of this statement comes from the fact, that $CV_2 ^{red}$ is homoeomorphic to the plane, hence topologically we can't determine faces. We will later see a proof of this statement for $CV_n^{red}$ for arbitrary $n \geq 2$ in theorem \ref{4 distinguish faces in CVn}.


An important fact about geodesics in Outer Space is, that a lot of information is already stored in witnesses. The first fact is, that we won't loose witnesses along a geodesic as shown in the following lemma:

\begin{Lem}\label{keep witness}
Let $g: I \to CV_n$ be a geodesic from $A$ to $B$, $\gamma \in F_n, t \in I$ and $C=g(t)$.\\
Then $\gamma \in F_n$ is a witness from $A$ to $B \iff \gamma$ is a witness from $A$ to $C$ and from $C$ to $B$.\\
The same holds for the symmetric case.
\end{Lem}

\begin{Bew}
Wlog. assume all of the graphs to be normalized.\\
\glqq $\Rightarrow$\grqq : Assume $\gamma$ is a witness from $A$ to $B$ but not a witness from $A$ to $C$ or from $C$ to $B$, hence  $\Lambda_R(A,C)>\frac{l_C(\gamma)}{l_A(\gamma)}$ or $\Lambda_R(C,B)>\frac{l_B(\gamma)}{l_C(\gamma)}$. Now $C$ is an intermediate point of a geodesic, hence
\begin{align*}
d_R(A,B)&=d_R(A,C)+d_R(C,B)\\
&=\log(\Lambda_R(A,C))+\log(\Lambda_R(C,B))\\
&< \log(\frac{l_C(\gamma)}{l_A(\gamma)})+\log(\frac{l_B(\gamma)}{l_C(\gamma)})\\
&=\log(\frac{l_C(\gamma)}{l_A(\gamma)}\frac{l_B(\gamma)}{l_C(\gamma)})=\log(\frac{l_B(\gamma)}{l_A(\gamma)})= d_R(A,B)
\end{align*}
which is the desired contradiction.\\
\glqq $\Leftarrow$\grqq : Let $\gamma$ be a witness from $A$ to $C$ and from $C$ to $B$. Again $C$ is an intermediate point of a geodesic and so:
\[d_R(A,B)=d_R(A,C)+d_R(C,B)=\log(\frac{l_C(\gamma)}{l_A(\gamma)})+\log(\frac{l_B(\gamma)}{l_C(\gamma)})= \log(\frac{l_B(\gamma)}{l_A(\gamma)})\]
Hence $\Lambda_R(A,B)=\frac{l_B(\gamma)}{l_A(\gamma)}$ and so $\gamma$ is a witness from $A$ to $B$.
\end{Bew}
Analogue we get the symmetric case by applying lemma \ref{keep witness} to each direction.

On the other hand we can glue any geodesics together, if the endpoints have fitting witnesses.

\begin{Lem}\label{glueing geodesics}
Let $A,B,C \in CV_n$ and $g,h$ be (symmetric) geodesics from $A$ to $C$ respectively $C$ to $B$. Then the following are equivalent:
\begin{enumerate}[(i)]
\item $g*h$ (=concatenation of $g$ and $h$) is a geodesic from $A$ to $B$
\item the set of witnesses from $A$ to $B$ is the intersection of the witnesses from $A$ to $C$ and the witnesses from $C$ to $B$, i.e. $W_{(R)}(A,B) = W_{(R)}(A,C) \cap W_{(R)}(C,B)$.

\item $\exists \ \gamma \in F_n^2$ (resp. $F_n$) s.t. $\gamma$ is a witness from $A$ to $C$ and a witness from $C$ to $B$, i.e. $W_{(R)}(A,C) \cap W_{(R)}(C,B) \neq \emptyset$.
\item There exists a (symmetric) geodesic $f$ from $A$ to $B$ s.t. $C$ lies on $f$.
\item $d_{(R)}(A,B)=d_{(R)}(A,C)+d_{(R)}(C,B)$
\end{enumerate}
\end{Lem}

\begin{Bew}
We will prove only the asymmetric case. The symmetric case follows then directly by interchanging $A$ and $B$.\\
\glqq $(i) \Rightarrow (ii)$\grqq: Follows directly from lemma \ref{keep witness}\\
\glqq $(ii) \Rightarrow (iii)$\grqq: Clear, since there exists at least one witness from $A$ to $B$.\\
\glqq $(iii) \Rightarrow (v)$\grqq: Let $\gamma \in F_n$ be as stated in (iii), then we have
\begin{align*}
d_R(A,B)&\geq \log(\frac{l_B(\gamma)}{l_A(\gamma)})\\
&=\log(\frac{l_C(\gamma)}{l_A(\gamma)})+\log(\frac{l_B(\gamma)}{l_C(\gamma)})\\
&=d_R(A,C)+d_R(C,B) \geq d_R(A,B)
\end{align*}
Where last equality holds since $\gamma$ is a witness from $A$ to $C$ and from $C$ to $B$. Hence the desired equality holds.\\
\glqq $(v) \iff (i)$\grqq: Follows directly from lemma \ref{2 lemma triangle ineq and geodesics} and that $g$ and $h$ are geodesics. Namely assume $(v)$ holds, then for points $a,b$ on $g$ and $c$ on $h$ we get by the triangle inequality and lemma \ref{2 lemma triangle ineq and geodesics}
\begin{align*}
d(A,B)&= d(A,C)+d(C,B)\\
&=d(A,a)+d(a,b)+d(b,C) + d(C,c) + d(c,B)\\
&\geq d(A,C)+d(a,b)+d(b,c)+d(c,B) \\
&\geq d(A,C)+d(a,c)+d(c,b)\\
& \geq d(A,B)
\end{align*}
hence equality holds everywhere. Similar for different distributions of $a,b,c$ on $g \star h$\\
\glqq $(iv) \iff (i)$\grqq: $(i) \Rightarrow (iv)$ is clear and $(iv) \Rightarrow (v)$ follows from lemma \ref{2 lemma triangle ineq and geodesics}.
\end{Bew}
Keep in mind that $(i) \iff (iv) \iff (v)$ holds in every metric space.
\vspace{\baselineskip}

Lemma \ref{glueing geodesics} implies that a point $C \in CV_n$ lies on a geodesic from $A$ to $B$ if and only if one (and hence all) witness from $A$ to $B$ is also a witness from $A$ to $C$ and from $C$ to $B$. For the symmetric metric this is not sufficient, since there might be no symmetric geodesic from $A$ to $B$ at all (s. example \ref{non_geodesic_ball_dim2}).
%

Another interesting aspect is, that the corresponding witnesses to a finite geodesic can be considered as a \emph{coarse direction} the geodesic has. This is in a certain way the only information a geodesic remembers from its past and cares about, if you want to continue it. Lemma \ref{keep witness} tells us, that these coarse direction will be kept during the whole geodesic and lemma \ref{glueing geodesics} means, that if we want to continue a finite geodesic, we only have to care about the witness of its two endpoints.

In particular since there exists always a candidate witness and there are finitely many candidates for a point we can assign to an outgoing geodesic ray a (global) coarse direction, namely a candidate witness from the origin to every point on the geodesic.

Using that locally minimizing geodesics are piecewise minimizing geodesics glued together as in lemma \ref{glueing geodesics} we can now construct two pathogens of minimizing geodesics also occurring in $\R^n$ with the maximum norm.

\begin{Bem}\label{2 fun stuff with locally-minimizing geodesics}
\begin{enumerate}[(i)]
\item There exist null-homotopic locally minimizing geodesics. One can can easily construct such an example by iteratively changing the coarse direction. For example let $\Gamma :=$ \begin{minipage}{2cm}
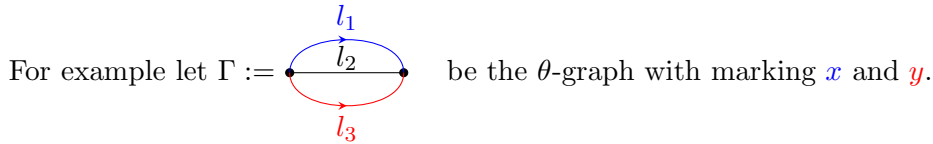
\begin{tikzpicture}[scale=1.5]
\filldraw (0,0) circle[radius=1pt];
\filldraw (1,0) circle[radius=1pt];
\draw(0,0) to node[above=-0.1]{$l_2$} (1,0);
\draw[ blue,postaction={decorate,decoration={
		markings, mark=at position .5 with {\arrow{stealth}} }}] (0,0) to [out=90,in=90] node[above]{$l_1$}  (1,0);
\draw[red,postaction={decorate,decoration={
		markings, mark=at position .5 with {\arrow{stealth}} }}] (0,0) to [out=270,in=270] node[below]{$l_3$} (1,0);
\end{tikzpicture}
\end{minipage} be the $\theta$-graph with marking $\textcolor{blue}{x}$ and $\textcolor{red}{y}$.

Consider the three points $A,B,C \in CV_2$ corresponding to $\Gamma$ with edge-lengths $(l_1,l_2,l_3) = (1,1,1), (2,1,1)$ and $(1,1/3,1)$. 

Observe that the candidates of $\Gamma$ are $x,y$ and $xy^{-1}$, hence for each pair of $A,B,C$ at least one of them is a witness. A short calculation shows $\{x,xy^{-1}\} \subset W_R(A,B),\linebreak
\{y,xy^{-1}\} \subset W_R(B,C)$  and $\{x,y\} \subset W_R(C,A)$. Hence by lemma \ref{glueing geodesics} concatenating the straight edges $\overrightarrow{AB}, \overrightarrow{BC}$ and $\overrightarrow{CA}$ will yield a closed, locally minimizing geodesic, since any two edges have a common coarse direction. Similarly one can construct an example for the symmetric metric in the shape of a hexagon.
\item Let $g: I \to CV_n$ be a continuous path and $\varepsilon > 0$.

Then there exists a locally minimizing geodesic $h: J \to CV_n$, s.t. $g$ lays in the $\varepsilon$-neighbourhood of $h$. In particular there exist locally minimizing geodesics which are dense in $CV_n$. A technical complete proof which uses some elements from the following chapters can be found in the appendix.
\end{enumerate}
\end{Bem}

\section{Envelopes in Outer Space}
As mentioned before, geodesics in Outer Space are almost never unique. To give a measure how much uniqueness fails, it seems reasonable to look at the whole set of geodesics at the same time. We borrow the notion of \emph{envelope} from \cite{DLRT}.

\begin{defi}
Let $(X,d)$ be a metric space and $A,B \in X$.\\
Then we define the \emph{envelope from $A$ to $B$} as
\begin{align*}
\Env_d(A,B):=\{C \in X \ | \ &\exists \text{ geodesic } g: [0,1] \to X , t\in [0,1], \\
&\text{ s.t. } g(0)=A, g(1)=B, g(t)=C\}.
\end{align*} 
For the envelopes in $CV_n$ we will write $\EnvR:=\Env_{d_R}$ for the asymmetric metric and $\Env:=\Env_d$ for the symmetric metric.
\end{defi}

Envelopes have the following important and easy to see properties which we will use later on.

\begin{Bem}
\begin{enumerate}[(i)]
\item It is clear, that isometries send envelopes to envelopes since isometries have to send geodesics to geodesics.
\item The diameter of an envelope is bounded. More explicitly:
\begin{itemize}
\item if the metric is symmetric the diameter of an envelope is the distance of the two endpoints by the triangle inequality: For $x,y \in \Env_d(A,B)$ we have
\begin{align*}
&d(x,y) \leq d(x,B)+d(B,y) \text{ and } d(x,y) \leq d(A,x)+d(A,y)\\
\Rightarrow &2*d(x,y) \leq 2*d(A,B)
\end{align*}
\item In an asymmetric metric we get for all $x,y \in \Env_d(A,B)$
\[d(x,y) \leq d(x,B)+d(B,A)+d(A,y) \leq 2*d(A,B)+d(B,A).\]
Consider as an example an oriented graph with vertices $A,B$, two equally long edges from $A$ to $B$ and one edge from $B$ to $A$ to see this is the best estimate we can do in the general setting.
\end{itemize}
\item By the equivalence lemma \ref{glueing geodesics} (i) $\iff$ (v) we have nested envelopes, i.e. for $A,B \in X, C \in \Env(A,B)$ we have
\[\Env(A,C) \cup \Env(C,B) \subseteq \Env(A,B).\]
\item By lemma \ref{glueing geodesics} (v) we can also write envelopes in $CV_n$ as
\[\Env_R(A,B)=\{C \in CV_n \ | \ d_R(A,B)=d_R(A,C)+d_R(C,B)\}.\]
\end{enumerate}
\end{Bem}

Before we look into envelopes in $CV_n$ we will introduce two notations:

\begin{Not}
Let $A \in CV_n$.
\begin{enumerate}[(i)]
\item We denote from now on $T(A)$ as the simplex in $CV_n$ coming from the topological type of $A$, i.e. all elements of $CV_n$ gained by changing only the lengths of $A$.
\item Let $e \in E(A)$ be an edge in the underlying graph of $A$ and $\gamma \in F_n$. Then we denote by $\#(e,\gamma)$ the number of times the cyclically reduced path corresponding to $\gamma$ passes through $e$ (without considering the orientation). This depends only on the topological type of $T(A)$ and not its lengths.
\item The supporting simplices of an envelope denotes all simplices $T \subset CV_n$ with non-empty intersection
\[\supp(\EnvR(A,B)):=\{\Delta \subset CV_n \ | \ \Delta \text{ is an open simplex and } \Delta \cap \EnvR(A,B) \neq \emptyset\}\]
\end{enumerate}
\end{Not}

Envelopes in $CV_n$ behave nicely in regard of the simplicial structure of $CV_n$. First of all, they are polytopes in each simplex as we will see in lemma \ref{wiggling polytope is polytope}. We will use this to construct locally rigid geodesics. In the next chapter we will then use envelopes to determine faces of Outer Space.

\begin{Lem}\label{wiggling polytope is polytope}
$\EnvR(A,B)$ is a polytope, i.e. in each simplex an intersection of finitely many half-spaces.
\end{Lem}
 
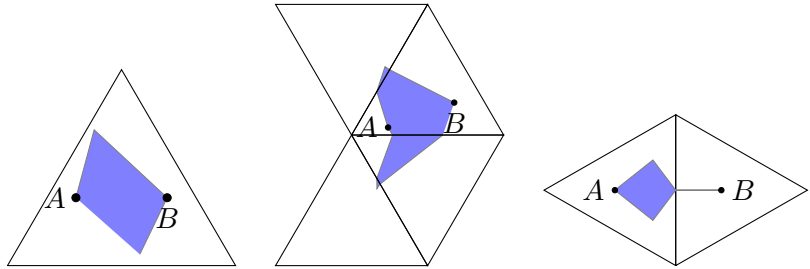
\begin{figure}[h]\label{bilder von envelopes}
\center
\begin{tikzpicture}[scale=3]
\draw (0,0)--(1,0) -- (0.5,0.8660254) -- (0,0);
\draw[gray, fill=blue!50] (0.3,0.3)  -- (0.38,0.6) -- (0.7,0.3) -- (0.58,0.05) ;
\draw[fill] (0.3,0.3) circle[radius=0.5pt] node[left]{$A$};
\draw[fill] (0.7,0.3) circle[radius=0.5pt] node[below]{$B$};
\end{tikzpicture}
\quad 
\begin{tikzpicture}[scale=2]
\draw[gray, fill=blue!50] (0.675,0.215) -- (0.25, 0.433) -- (0.22, 0.453) -- (0.166,0.288)-- (0.24,0.05) -- (0.27,0) -- (0.166, -0.288) -- (0.167,-0.355)-- (0.1875, -0.32475) -- (0.6,0);  
\draw[fill] (0.24,0.05) circle[radius=0.5pt] node[left]{$A$};
\draw[fill] (0.675,0.215) circle[radius=0.5pt] node[below]{$B$};
\draw (0,0)--(1,0) -- (0.5,0.8660254) -- (0,0);
\draw (0,0)--(1,0) -- (0.5,-0.8660254) -- (0,0);
\draw (0,0)-- (0.5,-0.8660254) -- (-0.5,-0.8660254) -- (0,0);
\draw (0,0)-- (0.5,0.8660254) -- (-0.5,0.8660254) -- (0,0);
\end{tikzpicture}
\quad
\begin{tikzpicture}[scale=2]
\draw (0,0)--(0,1) -- (-0.8660254,0.5) -- (0,0);
\draw (0,0)--(0,1) -- (0.8660254,0.5) -- (0,0);
\draw[gray, fill=blue!50] (-0.4,0.5)  -- (-0.15,0.7)-- (0,0.5) -- (0.3,0.5) -- (0,0.5) -- (-0.15,0.3) -- (-0.4,0.5);
\draw[fill] (-0.4,0.5) circle[radius=0.5pt] node[left]{$A$};
\draw[fill] (0.3,0.5) circle[radius=0.5pt] node[right]{$B$};
\end{tikzpicture}
\caption{Some envelopes $\EnvR(A,B)$ in $CV_2$} 
\end{figure}

\begin{Bew}
Let $\gamma$ be a witness from $A$ to $B$. By the implication after lemma \ref{glueing geodesics} we have $C \in \EnvR(A,B) \iff \gamma $ is a witness from $A$ to $C$ and from $C$ to $B$. So each $\omega \in cand(A)$ and $\delta \in cand(C)$ yields a linear inequality and hence parametrises a half-space in the simplex $T(C)$:
\begin{align*}
\frac{l_C(\gamma)}{l_A(\gamma)}\geq \frac{l_C(\omega)}{l_A(\omega)} &\iff 
\sum_{e_i \in E(C)} l_C(e_i) \cdot  (l_A(\omega) \cdot \#(e_i, \gamma)- l_A(\gamma) \cdot \#(e_i, \omega)) \geq 0 \quad \quad (\star)\\
\frac{l_B(\gamma)}{l_C(\gamma)}\geq \frac{l_B(\delta)}{l_C(\delta)} &\iff 
\sum_{e_i \in E(C)} l_C(e_i) \cdot (l_B(\gamma) \cdot \#(e_i, \delta) - l_B(\delta) \cdot \#(e_i, \gamma)) \geq 0 \quad \quad (\star \star)
\end{align*}
Since the terms $(l_A(\omega) \cdot \#(e_i, \gamma)- l_A(\gamma) \cdot \#(e_i, \omega))$ and $(l_B(\gamma) \cdot \#(e_i, \delta) - l_B(\delta) \cdot \#(e_i, \gamma))$ do not depend on the weights we get indeed a linear inequality and hence a half-space in $T(C)$.

By theorem \ref{1 candidate is enough} there exists a maximally stretched candidate, hence $\gamma$ is maximally stretched and thus a witness from $A$ to $C$ and from $C$ to $B$ if and only if above inequalities are satisfied for all $\omega \in cand(A), \delta \in cand(C)$,
i.e. if $C$ lies in the intersection of all these half-spaces. These are finitely many and therefore we are done.
\end{Bew}

\begin{Bem}
In the following, when we talk about the faces of an envelope we only mean faces arising from equalities in the $(\star)$ or $(\star \star)$-inequalities above.
In particular that means we will exclude the faces arising solely from intersections with a simplex. 
\end{Bem}

On the other hand by the following lemma the support of an envelope is finite, since its diameter is bounded.

\begin{Lem}
For any $B \in CV_n$ and $r>0$ the ingoing ball $B_r^{in}(B):=\{A \in CV_n \ | \ d_R(A,B) \leq r \}$ intersects with only finitely many simplices of $CV_n$.
\end{Lem}

\begin{proof}
Wlog. we can assume that $B$ is the standard marked rose $R_n$ with edge lengths all $1/n$, since we have $B_r^{in}(B) \subseteq B_{r+d_R(B,R_n)}(R_n)$. We will show that for a fixed graph there are only finitely many markings  such that it still lies in $B_r(R_n)$. Since there are only finitely many finite graphs with valency greater three and fundamental group $F_n$ the claim follows.

Fix for a normalized element $A \in CV_n$ a spanning tree and label the edges outside of the tree with words $w_1, \dots, w_n \in F_n$ according to its marking. Since each $w_i$ corresponds to a simple loop it has length of at most 1 in $A$. On the other hand for each word $w \in F_n$ its length in $B$ is $\frac{1}{n}$th of its cyclically reduced word-length. If now $A \in B_r^{in}(B)$ we have $\log \frac{l_B(w_i)}{l_A(w_i)} \leq r$, hence cyclically reduced edge labels of $A$ have at most word-length $n e^r$.

On the other hand assume $w_i$ is not cyclically reduced. Since we can simultaneously conjugate all labels in $A$ without changing the element in $CV_n$ we can assume, that there exists a $w_j$ such that $w_i w_j$ is a cyclically reduced word. The length of $w_i w_j$ in $A$ is at most $2$. We have now to consider two cases:

\begin{itemize}
\item  $w_j$ is not cyclically reduced: Then $w_i w_j$ is already a reduced word and hence its cyclically word-length is just the sum of the word-lengths of $w_i$ and $w_j$. Thus $w_i$ has word-length of at most $2n e^r$.
\item $w_j$ is cyclically reduced: Then the reduced word length of $w_i w_j$ is at least the difference of the  word-lengths of $w_i$ and $w_j$. Since $w_j$ is cyclically reduced its word-length is at most $n e^r$ and thus $w_i$ has at most word-length of $3n e^r$.
\end{itemize}
\end{proof}

\begin{Kor}
$\EnvR(A,B)$ is compact.
\end{Kor}

\begin{Bew}
Since the diameter of an envelope is bounded, it stays away from missing faces in $CV_n$ and has non-empty intersection with at most finitely many simplices. Hence the intersection of the envelope with a simplex is closed in the simplicial closure and therefore compact. Thus the envelope is as finite union of compact sets again compact.
\end{Bew}

As we have seen in lemma \ref{wiggling polytope is polytope} each envelope is the intersection of two cones coming from the two end-points. Namely the intersection of half-spaces belonging to the inequalities of type $(\star)$ and the inequalities of type $(\star \star)$. One can see these cones as set of points a geodesic ray can reach, if we fix the coarse direction and the start-point.

\begin{defi}
Let $S \subseteq F_n$ be a subset which we consider as coarse direction or wanted witnesses. Then we call
\[\Env_R^{out} (A,S):=\{B \in CV_n \ | \ S \subseteq W_R(A,B)\}\]
the \emph{out-envelope of $A$ in the direction of $S$} and
\begin{align*}
\Env_R^{in} (B,S)&:=\{A \in CV_n \ | \ S \subseteq W_R(A,B)\}
\end{align*}
the \emph{in-envelope of $B$ in the direction of $S$}.\\
If $S=\{\gamma\}$ is a singleton, we will just write $\Env_R^{out}(A,\gamma)$.
\end{defi}

\begin{Bem} \label{3 remark basic properties out-envelopes}
\begin{enumerate}[(i)]
\item By lemma \ref{glueing geodesics} the in- and out-envelopes tell, how one can extend geodesics in either direction and furthermore we get
\[\EnvR(A,B)=\Env_R^{out} (A,S) \cap \Env_R^{in} (B,S)\]
for all non-empty $S \subseteq W_R(A,B)$.
\item As in emma \ref{wiggling polytope is polytope} we get that in- and out-envelopes are polytopes in each simplex, namely the out-envelopes are parametrized by the inequalities of $(\star)$ and the in-envelopes by the $(\star \star)$-inequalities of \ref{wiggling polytope is polytope}. 
\item By definition intersections of out- resp. in-envelopes are the out- resp. in-envelopes of their union of directions, i.e. $\Env_R^{out}(A,S)=\bigcap_{\gamma \in S} \Env_R^{out}(A,\gamma)$.
\end{enumerate}
\end{Bem}

By Remark \ref{c2_geodesics_in_CVn} we know, that almost never two points are joined by unique geodesics, but with the help of envelopes we can construct a geodesic, which is at least piecewise unique. 

\begin{Satz}\label{unique geodesic segments}
For each $A,B \in CV_n$ there exist geodesic segments $g_1, \dots, g_l$, s.t. for each $i$ $g_i$ is the unique asymmetric geodesic joining its two endpoints and $g=g_1 * g_2 * \dots *g_l$ is an asymmetric geodesic from $A$ to $B$.
\end{Satz}

In the proof we will construct the geodesic segments starting from $A$ by moving along edges of envelopes. This yields unique geodesic segments by the following lemma:

\begin{Lem}\label{3 Lemma not leaving hyperplanes and unique geodesics}
\begin{enumerate}[(i)]
\item Let $A,B \in CV_n$ and $A' \in \EnvR(A,B)$ on a hyperplane $H$, which comes from an equality of the form $(\star)$ in lemma \ref{wiggling polytope is polytope}, hence by a candidate $\omega \in \cand(A)$. In other words $H$ is a face of the out-envelope $\Env_R^{out}(A,W_R(A,B))$.

Then by lemma \ref{keep witness} each geodesic $g$ from $A$ to $A'$ must lay completely in $H$, since each point on $g$ has also $\omega$ as a witness from $A$.

The same holds for hyperplanes coming from $(\star \star)$, i.e. faces of the in-envelope of $B$, and geodesics to $B$.

In other words a geodesic from $A$ to $B$ can't enter a $(\star)$-hyperplane or leave a $(\star \star)$-hyperplane.
\item All edges of an envelope $\EnvR(A,B)$, which have as a coarse direction (a subset of) $W_R(A,B)$, are unique geodesic segments. In particular all emanating edges from $A$ and incoming edges to $B$ are unique geodesics.
\item Let $S \subset F_n$, then consecutive edges in $\Env_R^{out}(A,S)$ form rigid geodesics. Similarly consecutive edges in $\Env_R^{in}(B,S)$ form rigid geodesics.
\end{enumerate}
\end{Lem}

\begin{Bew}
(ii): For an edge $e$ with a coarse direction in $W(A,B)$ let $A'$ and $B'$ be the endpoints of $e$ and $g$ any geodesic from $A'$ to $B'$ (for example $e$ itself is such a geodesic). By lemma \ref{glueing geodesics} we can extend $g$ to a geodesic $\tilde{g}$ from $A$ to $B$.

Since $e$ is an edge, we can write it as an intersection of hyperplanes $H_A$ coming from $(\star)$ and hyperplanes $H_B$ from $(\star \star)$. By previous (i) $\tilde{g}$ lies at least until $B'$ in all hyperplanes $H_A$ and from $A'$ to $B$ in all hyperplanes $B'$. In particular $g$ lies in all hyperplanes $H_A$ and $H_B$ and hence is $e$.\\
(iii): Let $e_1, \dots, e_n$ be consecutive edges in $\Env_R^{out}(A,S)$ and $A_0, A_1, \dots, A_n$ their endpoints. As before each geodesic $g$ from $A$ to $A_n$ must be contained in all hyperplanes containing $A_n$, therefore $g$ must contain $e_n$ and hence $A_{n-1}$ lies on $g$. Inductively $g$ has to go through $A_0, \dots, A_n$ over $e_1 * \dots * e_n$ and therefore $e_1 * \dots * e_n$ is the unique geodesic from $A_0$ to $A_n$.
\end{Bew}

\begin{Bew}[of Theorem \ref{unique geodesic segments}]
Let $A,B \in CV_n$ any points. We will now construct $A_i, E_i$ and $g_i$ for $i \in \N$ inductively  starting with $A_0:=A$ and $E_0:=\EnvR(A,B)$:

Starting at $A_i$ choose any consecutive edges $e_1, \dots, e_{k_i}$ in $\Env_R(A_i,B)$ until they hit the first time a new hyperplane coming from an equality of type $(\star \star)$ and denote this point bei $A_{i+1}$. This means that $e_1, \dots, e_{k_i}$ are actually edges of $\Env_R^{out}(A_i,W(A_i,B))$ (more exactly $e_{k_i}$ might be only a part of an edge) and thus by lemma \ref{3 Lemma not leaving hyperplanes and unique geodesics} \linebreak
$g_{i+1}:=e_1 * \dots * e_{k_i}$ is the unique geodesic from $A_i$ to $A_{i+1}$. Since there are only finitely many edges in $\Env_R(A_i,B)$ each such a sequence of edges is finite.

By lemma \ref{3 Lemma not leaving hyperplanes and unique geodesics} we have
\[CW(A_{i+1},B) \subsetneq CW(A_i,B) \subseteq \bigcup_{\substack{\Delta \in \supp(\\ \EnvR(A,B))}} \cand(\Delta)\]
which is a finite set, so previous induction stops after finitely many steps.
\end{Bew}

Since between any two points we can choose a candidate witness, we can fix a point and cover $CV_n$ with out-envelopes. In particular this will be a partition of each simplex into polytopes and each out-envelope can be seen as a face of such a polytope.

\begin{Prop}\label{3 basic properties out-envelopes}
Let $A \in CV_n$, then we have:
\begin{enumerate}[(i)]
\item $CV_n = \bigcup\limits_{\gamma \in \cand(A)} \Env_R^{out}(A,\gamma)$
\item $\{A\}= \bigcap\limits_{\gamma \in \cand(A)} \Env_R^{out}(A,\gamma)$
\item For all $\gamma \in \cand(A)$ the interior $\Env_R^{out}(A,\gamma)^O$ is non-empty.
\item For all sets of witnesses $M \subseteq F_n$ and simplices $\Delta \in \supp(\Env_R^{out}(A,M))$, there exists a subset of candidates $S \subseteq \cand(A)$ such that their out-envelopes are the same in $\Delta$, namely $\Delta \cap \Env_R^{out}(A,M)=\Delta \cap \Env_R^{out}(A,S)$.
\item Let $S_1, S_2 \subseteq F_n$ and $\Delta$ a simplex in $CV_n$, then $\Env_R^{out}(A, S_1) \cap \Env_R^{out}(A, S_2) \cap \Delta$ is a face of $\Env_R^{out}(A, S_1) \cap \Delta$.
\end{enumerate}
\end{Prop}

\begin{Bew}
\begin{enumerate}[(i)]
\item follows directly from the fact, that for each $B \in CV_n$ there exists a candidate witness $\gamma \in CW(A,B)$.\\
\item Since for each $B \in \Env_R^{out}(A, S)$ there exists a geodesic $g \subseteq \Env_R^{out}(A, S)$ from $A$ to $B$ it is enough to show the statement for any neighbourhood of $A$. 

Let now $B \in \bigcap\limits_{\gamma \in \cand(A)} \Env_R^{out}(A,\gamma) \cap U$ for a small enough neighbourhood $U$ of $A$ and $\gamma \in CW(B,A)$.
Assume $\gamma$ is not a candidate of $A$. If we choose $U$ enough small such we have that the topological type of $A$ is the same of $B$ up to contracting some edges. This means that $\gamma$ is of the form:

\begin{figure}[h]
\center
\begin{tikzpicture}
\node at (0.5,0.6) {$\gamma$};
\fill(0,0.1) circle[radius=1pt];
\fill (1,0.1) circle[radius=1pt];
\draw (0,0.1) to[out=180, in=-70] (-0.1,0.2) to[out=110, in=-270] (-1,0);
\draw (0,0.1) to[out=0, in=70] (0.1,0.2) to[out=70, in=110] (0.9,0.2) to[out=110, in=180] (1,0.1) to[out=0, in=70] (1.1,0.2);
\draw[green] (0,0.1) -- (0,0);
\draw[green] (1,0.1) -- (1,0);
\begin{scope}[yscale=-1]
\fill(0,0.1) circle[radius=1pt];
\fill (1,0.1) circle[radius=1pt];
\draw (0,0.1) to[out=180, in=-70] (-0.1,0.2) to[out=110, in=-270] (-1,0);
\draw (0,0.1) to[out=0, in=70] (0.1,0.2) to[out=70, in=110] (0.9,0.2) to[out=110, in=180] (1,0.1) to[out=0, in=70] (1.1,0.2);
\draw[green] (0,0.1) -- (0,0);
\draw[green] (1,0.1) -- (1,0);
\end{scope}

\node at (1.5,0) {$\dots$};

\begin{scope}[shift={(3,0)}, xscale=-1]
\draw[green] (0,0.1) -- (0,0);
\draw[green] (1,0.1) -- (1,0);
\fill(0,0.1) circle[radius=1pt];
\fill (1,0.1) circle[radius=1pt];
\draw (0,0.1) to[out=180, in=-70] (-0.1,0.2) to[out=110, in=-270] (-1,0);
\draw (0,0.1) to[out=0, in=70] (0.1,0.2) to[out=70, in=110] (0.9,0.2) to[out=110, in=180] (1,0.1) to[out=0, in=70] (1.1,0.2);

\begin{scope}[yscale=-1]
\draw[green] (0,0.1) -- (0,0);
\draw[green] (1,0.1) -- (1,0);
\fill(0,0.1) circle[radius=1pt];
\fill (1,0.1) circle[radius=1pt];
\draw (0,0.1) to[out=180, in=-70] (-0.1,0.2) to[out=110, in=-270] (-1,0);
\draw (0,0.1) to[out=0, in=70] (0.1,0.2) to[out=70, in=110] (0.9,0.2) to[out=110, in=180] (1,0.1) to[out=0, in=70] (1.1,0.2);

\end{scope}
\end{scope}

\draw[red,postaction={decorate,decoration={
		markings, mark=at position .5 with {\arrow{stealth}} }}] (0.2,0) arc (180:-180:0.3);
\node[red] at (0.5,0) {$\gamma_2$};
\draw[blue,postaction={decorate,decoration={
		markings, mark=at position .5 with {\arrow{stealth}} }}] (-0.8,0) arc (180:-180:0.3);
\node[blue] at (-0.5,0) {$\gamma_1$};
\end{tikzpicture}
\caption{A new candidate close to face.}
\end{figure}
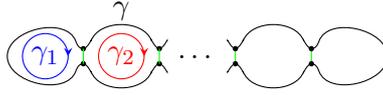

where the small green edges between the vertices are collapsed in $A$ and we might hide a barbell handle in the dots. But then $\gamma_1 \gamma_2 \in \cand(A)$ would be less stretched from $A$ to $B$ than $\gamma_1 \gamma_2^{-1}$ which contradicts that all candidates of $A$ are witnesses from $A$ to $B$.

If $B \in U \cap \Env_R^{out}(A, \cand(A))$ there exists a candidate maximally stretched from $B$ to $A$, which is also maximally stretched from $A$ to $B$ by assumption, hence we get $A=B$.
\item We will construct an open set contained in $\Env_R^{out}(A,\gamma)$.\\
If $A$ is in a maximal simplex, let $\varepsilon >0$ be smaller than each edge length of $A$ and $\gamma \in \cand(A)$. Let $B$ be any element in $CV_n$ obtained from $A$ by changing the edge lenthgs in the following way:

\begin{itemize}
\item Each edge not contained in $\gamma$ is shrinked by more than $\varepsilon$, i.e. $l_B(e) < l_A(e)-\varepsilon$.
\item If $\gamma$ is a simple loop and $n_\gamma$ is the number of edges in $\gamma$ each edge contained in $\gamma$ is stretched at most $\varepsilon/ n_\gamma$, i.e. $l_A(e) < l_B(e) < l_A(e)+\varepsilon/ n_\gamma$.
\item If $\gamma$ is a barbell, let denote $\gamma_1, \gamma_2$ the two circles of $\gamma$ and $\alpha$ the barbell handle as edge paths, i.e. looking at $\gamma$ as a sequence of edges we have $\gamma = \gamma_1 * \alpha * \gamma_2 * \alpha^{-1}$. Then each edge contained in one of the two circles is shrinked and the circles are shrinked by at most $\varepsilon/2$, i.e. $l_A(\gamma_i) - \varepsilon/2 < l_B(\gamma_i)$. For $n_\alpha$ the number of edges in $\alpha$ stretch each edge contained in $\alpha$  less than $\varepsilon/ n_\alpha$ such that $\alpha$ is stretched more than $\varepsilon/2$.
\end{itemize}

It is now easy to check, that $\gamma$ is the only and hence maximally stretched path from $A$ to $B$ and hence $B \in CV_n(A, \gamma)$ holds. Moreover the set of such $B$ is open, hence the claim follows.

Let now $A$ be not in a maximal simplex, then as before we can construct $A' \in \Env_R^{out}(A, \gamma)$ with $\gamma \in \cand(A')$ by relaxing vertices of valency bigger than 3 with $k$ edges of length $\varepsilon/k$ along $\gamma$. We only need to take care if $\gamma$ is a figure of eight to relax it into a barbell to get also $\gamma \in \cand(A')$. Since $\Env_R^{out}(A, \gamma) \subseteq \Env_R^{out}(A', \gamma)$ the claim follows.

\item By remark \ref{3 remark basic properties out-envelopes} (iii) we can assume wlog. that $M=\{\gamma\}$. Let $B$ a point in the relative interior of the polytope $\relint(\Env_R(A, \gamma) \cap \Delta)$. For $S := CW(A,B)$ we claim, that $\Env_R^{out}(A, \gamma) \cap \Delta = \Env_R^{out}(A,S) \cap \Delta$.

Let $B+V$ be the affine subspace spanned by $\Env_R^{out}(A, \gamma) \cap \Delta$ in $\Delta$ and $v \in V$ small enough such that $B+v, B-v \in \Env_R^{out}(A, \gamma) \cap \Delta$ holds.\\
For $\omega \in S$ the inequality in $(\star)$ is an equality for $B$, i.e. we have $\frac{l_B(\gamma)}{l_A(\gamma)} = \frac{l_B(\omega)}{l_A(\omega)}$. Since $B+v$ and $B-v$ are in $\Env_R^{out}(A, \gamma)$ the equality also holds for $B\pm v$ else $\omega$ would be more stretched than $\gamma$ for either $B+v$ or $B-v$. Thus the equality holds for all $B+V$, in other words for all $\omega \in S$ and $C \in B+V \supseteq \Env_R^{out}(A, \gamma)$ we have $\frac{l_C(\gamma)}{l_A(\gamma)} = \frac{l_C(\omega)}{l_A(\omega)}$ and hence $\Env_R^{out}(A, \gamma) \cap \Delta \subseteq \Env_R^{out}(A, S)$.

Let now $B+W$ be the affine subspace spanned by $\Env_R^{out}(A, S) \cap \Delta$ in $\Delta$ and $w \in W$ small enough.

If $C^{+}:=B+w \in (\Env_R^{out}(A, S) \setminus \Env_R^{out}(A, \gamma)) \cap \Delta$ then we have $\frac{l_{C^{+}}(\gamma)}{l_A(\gamma)} < \frac{l_{C^{+}}(\omega)}{l_A(\omega)}$ for all $\omega \in S$ and hence for $C^{-}:=B-w$ by $(\star)$ we get $\frac{l_{C^{-}}(\gamma)}{l_A(\gamma)} > \frac{l_{C^{-}}(\omega)}{l_A(\omega)}$. But if we choose $|w|$ small enough this also holds for all $\omega \in \cand(A) \setminus S$ (s. lemma \ref{lem:looseCWinNeighbourhood}), but this contradicts $CW(A,C^-) \neq \emptyset$.

\item Follows directly from (iv) and the inequalities $(\star)$ in lemma \ref{wiggling polytope is polytope}, since the defining half spaces of the polytope come from the candidates.
\end{enumerate}
\end{Bew}

\begin{Bem}
Statement (iii) doesn't hold in the reduced Outer Space, since we might not be able to relax a figure of eight into a barbell. For example in $CV_2^{red}$ the rose $A$ and $\gamma$ a figure of eight in $A$ has as outgoing envelope a unique geodesic.
\end{Bem}

We get similar results for the in-envelopes:

\begin{Prop}\label{3 basic properties in-envelopes}
Let $A \in CV_n$ and $S :=\{\gamma \in F_n \ | \ \gamma \text{ can be extended to a free generating set of } F_n\}=\bigcup_{A \in CV_n} \cand(A)$ the set of possible candidates, then we have:
\begin{enumerate}[(i)]
\item $CV_n = \bigcup\limits_{\gamma \in S} \Env_R^{in}(A,\gamma)$
\item $\{A\} = \bigcap\limits_{\gamma \in \cand(A)} \Env_R^{in}(A,\gamma)$
\item For all $\gamma \in S$ the interior $\Env_R^{in}(A,\gamma)^O$ is non-empty.
\item For all sets of witnesses $M \subseteq F_n$ and simplices $\Delta \in \supp(\Env_R^{out}(A,M))$, there exists a subset of candidates $S' \subseteq \cand(A)$ such that $\Delta \cap \Env_R^{in}(A,M)=\Delta \cap \Env_R^{in}(A,S')$.
\item Let $S_1, S_2 \subseteq F_n$ and $\Delta$ a simplex in $CV_n$, then $\Env_R^{in}(A, S_1) \cap \Env_R^{in}(A, S_2) \cap \Delta$ is a face of $\Env_R^{in}(A, S_1) \cap \Delta$.
\end{enumerate}
\end{Prop}

\begin{Bew}
All statements except (iii) are proven as in propostion \ref{3 basic properties out-envelopes}. 

For (iii) let $\gamma \in S$ and $(\gamma, \gamma_2, \dots, \gamma_n)$ a generating set of $F_n$. Consider the marked rainbow graph $B$

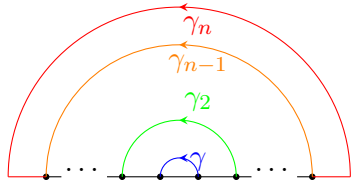
\begin{figure}[h]
\center
\begin{tikzpicture}
\filldraw (-2,0) circle[radius=1pt];
\node at (-1.5,0.1) {$\dots$};
\filldraw (-1,0) circle[radius=1pt];
\filldraw (-0.5,0) circle[radius=1pt];
\filldraw (0,0) circle[radius=1pt];
\filldraw (0.5,0) circle[radius=1pt];
\node at (1,0.1) {$\dots$};
\filldraw (1.5,0) circle[radius=1pt];
\draw[red] (-2.5,0) -- (-2,0);
\draw (-2,0) -- (-1.8,0);
\draw (-1.2,0) -- (0.7,0);
\draw[red] (1.5,0) -- (2,0);
\draw (1.3,0) -- (1.5,0);

\draw[red, postaction={decorate,decoration={
		markings, mark=at position .5 with {\arrow{stealth}} }}] (2,0) arc (0:180:9/4);
\node[red] at (0,2) {$\gamma_n$};
\draw[orange, postaction={decorate,decoration={
		markings, mark=at position .5 with {\arrow{stealth}} }}] (1.5,0) arc (0:180:7/4);
\node[orange] at (0,1.5) {$\gamma_{n-1}$};
\draw[green, postaction={decorate,decoration={
		markings, mark=at position .5 with {\arrow{stealth}} }}] (0.5,0) arc (0:180:3/4);
\node[green] at (0,1) {$\gamma_2$};
\draw[blue, postaction={decorate,decoration={
		markings, mark=at position .5 with {\arrow{stealth}} }}] (0,0) arc (0:180:1/4);
\node[blue] at (0,0.2) {$\gamma$};
\end{tikzpicture}
\caption{marked rainbow graph}
\end{figure}

where the edges belonging to $\gamma$ have length $< \varepsilon$ and all the other edges length $>1$. Hence all candidates of $B$ except $\gamma$ have at least length 2 and $\gamma$ has at most length $2 \varepsilon$. Since there are only finitely many candidates and there must be a witness candidate, we can choose $\varepsilon$ small enough, such that the rainbow graphs which satisfy the length conditions lay in $\Env^{in}_R(A, \gamma)$. These rainbow graphs form an open set.
\end{Bew}

Since the out- and in-going envelopes are parametrized by the inequalities $(\star)$ and $(\star \star)$ in lemma \ref{wiggling polytope is polytope} we get that continuously varying $A$ we continuously vary the out- and in-envelopes in each simplex as long as we don't change the candidates and the envelope stays in the simplex. In particular for $A,B \in CV_n$ in maximal simplices we can always choose neighbourhoods small enough, that we don't get intersections with new envelopes, i.e. we don't get additional candidate witnesses as the next lemma states.

\begin{Lem}\label{lem:looseCWinNeighbourhood}
Let $A,B \in CV_n$ and $S:=\{\gamma \in \cand(\Delta) \ | \ \Delta \text{ is a simplex in } CV_n \text{ and } A \in \overline{\Delta}\}$ the candidates close to $A$.

Then there exist neighbourhoods $U_A \ni A$ and $U_B \ni B$ such that for all $A' \in U_A$ and $B' \in U_B$ we have $W_R(A',B') \cap S \subseteq W_R(A,B) \cap S$.

In particular if $A$ lies in a maximal simplex we get $CW(A',B')\subseteq CW(A,B)$.
\end{Lem}

\begin{proof}
Considering the $(\star)$-inequalities of lemma \ref{wiggling polytope is polytope} and a $\gamma \in W_R(A,B)$ we have $\omega \in W_R(A,B)$ if and only if the $(\star)$-equality holds for $\omega$. We can choose $U_A$ small enough, such that $\cand(A')\subseteq S$ for all $A' \in U_A$ and so we can consider the finite set $S$ instead of $\cand(A')$ in the $(\star)$-inequalities. Varying $A$ continuously varies the coefficients of $l_B(e_i)$ in the $(\star)$-inequality and since $S$ is finite we can choose $U_A$ and $U_B$ in such a matter, that if for any $\omega \in S$ the sum is strictly greater than 0, it stays so for all $A' \in U_A$ and $B' \in U_B$. Therefore such $\omega$ is less stretched from $A'$ and $B'$ than $\gamma$, hence we don't gain any new witnesses from $S$ and the claim follows.
\end{proof}

While lemma \ref{3 Lemma not leaving hyperplanes and unique geodesics} (iii) tells us, that consecutive edges of out- and in-envelope are rigid geodesics, we will see that in fact all rigid geodesics are of that form. To see that we will need the following lemma.

\begin{Lem}\label{lem:interior_of_simplex_implies_many_geodesics}
Let $\bar{\Delta} \subset CV_n$ be a closed simplex in $CV_n$ and $A \in \bar{\Delta}, S \subseteq F_n$ such that $\mathcal{E}:=\Env_R^{out}(A,S)\cap \bar{\Delta}$ has dimension $k$.

Then for each relative interior point $B \in \relint(\Env_R^{out}(A,S))\cap \bar{\Delta}$ we have \linebreak
$\dim(\EnvR(A,B)\cap \bar{\Delta})=k $. In particular for $k\geq 2$ the geodesics from $A$ to $B$ is not unique.

The analogue statement holds for in-envelopes.
\end{Lem}

\begin{proof}
Let $B \in \relint(\Env_R^{out}(A,S))\cap \bar{\Delta}, S':=W_R(A,B)$ and $\mathcal{E}':=\Env_R^{out}(A,S')\cap \bar{\Delta}$. By $B \in \mathcal{E}$ we have $S \subseteq S'$ and thus $\mathcal{E}' \subseteq \mathcal{E}$.

Applying proposition \ref{3 basic properties out-envelopes} (v) on $\mathcal{E}$ and $\mathcal{E}'$ we get that $\mathcal{E}'$ is a face of $\mathcal{E}$. But $\mathcal{E}'$ contains an interior point $B$ of $\mathcal{E}$ and thus we have $\mathcal{E}'=\mathcal{E}$. 

Let now $U_A$ be as in lemma \ref{lem:looseCWinNeighbourhood} and $A' \in U_A \cap \mathcal{E}$. By $A' \in \mathcal{E}=\mathcal{E'}$ we get $S' \subseteq W_R(A,A')$ and by $A' \in U_A$ and lemma \ref{lem:looseCWinNeighbourhood} we have $CW(A',B) \subseteq S'$. But then lemma \ref{glueing geodesics} implies, that $A'$ lies on a geodesic from $A$ to $B$.

So if the dimension of $\mathcal{E}$ is $k$, then we have $\dim(\EnvR(A,B) \cap \bar{\Delta})=k$.

For $k \geq 2$ exists several points $A' \in \EnvR(A,B)$ with the same distance from $A$ which means they can't lie on the same geodesic.

Similar we get the statement for in-envelopes.
\end{proof}

Putting this together we get:

\begin{theorem}\label{theorem:unique_geodesics_are_edges}
Let $I  \subset \R$ be any interval, $t \in I$ and $g: I \to CV_n$ a geodesic. Then the following is equivalent:
\begin{enumerate}[(i)]
\item  $g$ is a rigid geodesic, i.e. $g|_{[a,b]}$ is the unique geodesic joining $g(a)$ and $g(b)$ for all $a \leq b \in I$
\item For all $t \in I$ $g|_{\geq t}$ lies on the edges of an out-envelope of $g(t)$ and $g|_{\leq t}$ lies on edges of in-envelopes of $g(t)$ (but not necessarily in a single in-envelope).
\end{enumerate}
\end{theorem}

\begin{Bew}
By lemma \ref{3 Lemma not leaving hyperplanes and unique geodesics} (iii) and setting $t=a$ we have $(ii) \Rightarrow (i)$.

Assume (ii) does not hold, then wlog $g|_{\geq t}$ does not lie on the edges of an out-envelope. Let $s\geq t$ be minimal such that $g|_{>s}$ lies not on edges of an out-envelope of $g(t)$. This implies $\Env_R^{out}(g(t),W_R(g(t),g(s')))$ has at least dimension 2 near $g(s)$ and $g(s')$ lies in the interior of the envelope for small enough $s'>s$. Fix such a small $s'>s$ such that $g(s)$ and $g(s')$ lie in the same closed simplex $\bar{\Delta} \subset CV_n$ and set $S:=W_R(g(t),g(s'))$. Since $s$ is minimal with this condition we get that each geodesic $g'$ from $g(t)$ to a point in $\Env_R^{out}(g(t),S)$ has coincide with $g$ until $s$. Thus we get $\Env_R^{out}(g(t),S)= g|_{[t,s]} \cup \Env_R^{out}(g(s),S) $ and $g(s')$ lies not on an edge of $\Env_R^{out}(g(s),S)$. By lemma \ref{lem:interior_of_simplex_implies_many_geodesics} to $g(s)$ and $g(s')$ we get different geodesics from $g(s)$ to $g(s')$ hence (i) also does not hold.
\end{Bew}

\begin{Bem}
The reason one out-envelope is enough for theorem \ref{theorem:unique_geodesics_are_edges} is that for a given geodesic $g: I \to CV_n$ and $t \in I$ we have by lemma \ref{keep witness} $CW(g(t),g(s_1)) \subseteq CW(g(t),g(s_2)) \neq \emptyset$ for all $t\leq s_1 \leq s_2$. Since $CW(g(t),g(s))$ is finite for all $s>t$ this implies $CW(g|_{>t}):=\bigcap_{s>t} CW(g(t),g(s)) \neq \emptyset$ and $CW(g(t),g(s))\subseteq CW(g|_{>t})$ for all $s>t$.

This argument does not hold for the negative direction $g|_{<t}$, since the candidates here depend on $g(s)$ and not $g(t)$. We might even get that $\bigcap_{s<t} W_R(g(s),g(t))$ is an empty set. An example for this in $CV_2^{red}$ can be constructed as follows. 

Enumerate the elements of $F_2$ as $(\gamma_i)_{i \in \N}$ and start with any simplex $\Delta \subset CV_2^{red}$. Let $A_0 \in \bar{\Delta}$ be a figure of eight. A short calculation shows that there exist two rigid geodesics from two different figures of eights $B_1,C_1 \in \bar{\Delta} \setminus T(A_0)$ ending in $A_0$. Choose $A_1:= B_1$ or $C_1$ such that we have $\gamma_1 \not \in W_R(A_1,A_0)$. Inductively choose $A_i$ such that $\gamma_i \not \in W_R(A_i,A_{i-1}) $ for all $i \in \N$. As in lemma \ref{3 Lemma not leaving hyperplanes and unique geodesics} (iii) this yields a rigid geodesic $g_i$ from $A_i$ to $A_0$ for all $i \in \N$ and thus $\gamma_i \not \in W_R(A_i,A_0)\subset W_R(A_i,A_{i-1})$. Hence after identifying the $g_i$ on their common parts we have constructed a rigid geodesic $g$ with $\bigcap_{s<t} W_R(g(s),g(t))=\emptyset$.
\end{Bem}

\section{Simplicial structure of $CV_n^{red}$}

We will now show how to distinguish faces in $CV_n^{red}$ by the use of envelopes. The important observation we use here is, that envelopes may have different dimensions in different simplices. In the reduced Outer Space we can construct such envelopes near any face as follows:

\begin{Lem}\label{4 geodesic_passing_face_looses_general}
Let $C \in CV_n^{red}$ be in a face, i.e. the underlying graph of $C$ has a vertex of degree at least 4. Let furthermore $U \ni C$ be a neighbourhood of $C$. Then there exist \linebreak
$A,B \in U$  such that the envelope $\EnvR(A,B)$ has near $A$ higher dimension than around $B$, i.e. there exist neighbourhoods $U_A$ of $A$ and $U_B$ of $B$ with $\dim(\EnvR(A,B)\cap U_A) > \dim(\EnvR(A,B)\cap U_B)$ (s. figure \ref{bilder von envelopes}). In particular we have $\dim(\EnvR(A',B')\cap U_A) \neq \dim(\EnvR(A',B')\cap U_B)$ for all $A' \in U_A$ and $B' \in B$.
\end{Lem}

\begin{Bew}
Since we can slightly relax vertices of valency 4 or greater, we can wlog. assume that $C$ has exactly one vertex $v$ of valency 4. Let's look at the star around $v$:
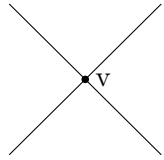
\[
\begin{tikzpicture}
\draw (-1,-1) -- (0,0)node[right]{v} node[circle,fill,inner sep=1pt]{};
\draw (-1,1) -- (0,0);
\draw (1,1) -- (0,0);
\draw (1,-1) -- (0,0);
\end{tikzpicture}
\]
Since $C \in CV_n^{red}$ there exists no separating edge and hence we find embedded circles $\textcolor{blue}{x}$ and $\textcolor{red}{y}$ passing through $v$:
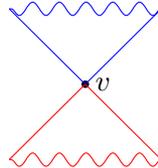
\[
\begin{tikzpicture}
\draw (0,0)node[right]{$v$} node[circle,fill,inner sep=1pt]{};
\draw[red] (-1,-1) -- (0,0) -- (1, -1);
\draw[red, decorate, decoration={snake}] (1, -1) -- (-1,-1);
\draw[blue] (-1,1) -- (0,0) -- (1,1);
\draw[blue, decorate, decoration={snake}] (1,1) -- (-1,1);
\end{tikzpicture}
\]
Furthermore we can assume that $x$ and $y$ are disjoint outside $v$ by cutting out common edges and glueing the paths back together:
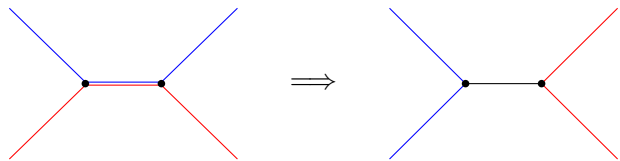
\[
\begin{tikzpicture}
\draw[blue] (-1,1) -- (0,0.02) -- (1, 0.02) -- (2,1);
\draw[red] (-1,-1) -- (0,-0.02) -- (1, -0.02) -- (2,-1);
\draw (0,0)  node[circle,fill,inner sep=1pt]{};
\draw (1,0)  node[circle,fill,inner sep=1pt]{};
\draw (3,0) node{$\Longrightarrow$};
\begin{scope}[shift={(5,0)}]
\draw[blue] (-1,1) -- (0,0) -- (-1,-1);
\draw[red] (2,1) -- (1, 0) -- (2,-1);
\draw (0,0)  node[circle,fill,inner sep=1pt]{};
\draw (1,0)  node[circle,fill,inner sep=1pt]{};
\draw (0,0) -- (1,0);
\end{scope}
\end{tikzpicture}
\]
Giving now $x$ and $y$ an orientation, we can see them as elements in $F_n$.\\
Let now $\varepsilon>0$ be small enough, then the following graphs are still inside $U$:
\begin{enumerate}
\item[$A:$] is obtained by relaxing $v$ to an edge $(v_1, v_2)$ of length $\varepsilon$ as in figure \ref{fig: relaxed_vertex_with_loops}
and keeping the rest as in $C$.
\item[$B:$] is obtained by relaxing $v$ in the different manner as in $A$ (see the orientation of $y$) and shrinking each edge of $C$ which does not lie in $x$ or $y$ by $2\varepsilon$, i.e. \linebreak
$l_B(e)=l_C(e)-2 \varepsilon$.
\end{enumerate}

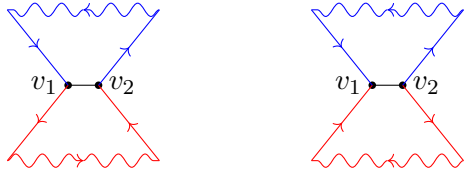
\begin{figure}[h]\label{fig: relaxed_vertex_with_loops}
\centering
\begin{tikzpicture}
\draw (0.2,0)node[right]{$v_2$} node[circle,fill,inner sep=1pt]{} -- (-0.2,0)node[left]{$v_1$} node[circle,fill,inner sep=1pt]{};
\draw[red, decoration={markings, mark=at position 0.5 with {\arrow{>}}},postaction={decorate}] (-0.2,0) -- (-1,-1); \draw[red, decoration={markings, mark=at position 0.5 with {\arrow{>}}},postaction={decorate}] (1, -1) -- (0.2,0);
\draw[red, ->, decorate, decoration={snake}] (-1,-1) -- (0,-1);
\draw[red, decorate, decoration={snake}] (0,-1) -- (1,-1);
\draw[blue, decoration={markings, mark=at position 0.5 with {\arrow{>}}},postaction={decorate}] (-1,1) -- (-0.2,0);
\draw[blue, decoration={markings, mark=at position 0.5 with {\arrow{>}}},postaction={decorate}] (0.2,0) -- (1,1);
\draw[blue, decorate, decoration={snake}]  (0,1) -- (-1,1);
\draw[blue, ->, decorate, decoration={snake}] (1,1) -- (0,1);
\begin{scope}[shift={(4,0)}]
\draw (0.2,0)node[right]{$v_2$} node[circle,fill,inner sep=1pt]{} -- (-0.2,0)node[left]{$v_1$} node[circle,fill,inner sep=1pt]{};
\draw[red, decoration={markings, mark=at position 0.5 with {\arrow{>}}},postaction={decorate}] (-1,-1) -- (-0.2,0); \draw[red, decoration={markings, mark=at position 0.5 with {\arrow{>}}},postaction={decorate}] (0.2,0) -- (1, -1);
\draw[red, decorate, decoration={snake}] (0,-1) -- (-1,-1);
\draw[red, ->,decorate, decoration={snake}] (1,-1) -- (0,-1);
\draw[blue, decoration={markings, mark=at position 0.5 with {\arrow{>}}},postaction={decorate}] (-1,1) -- (-0.2,0);
\draw[blue, decoration={markings, mark=at position 0.5 with {\arrow{>}}},postaction={decorate}] (0.2,0) -- (1,1);
\draw[blue, decorate, decoration={snake}]  (0,1) -- (-1,1);
\draw[blue, ->, decorate, decoration={snake}] (1,1) -- (0,1);
\end{scope}
\end{tikzpicture}
\caption{$v$ relaxed to get $A$ and $B$}
\end{figure}

Since topologically $A$ and $B$ are the same outside of $v$, each candidate is only stretched from $A$ to $B$ if it crosses $v_1$ or $v_2$ and stays inside of $x$ and $y$. Hence $xy$ is the only maximally stretched candidate from $A$ to $B$. Considering the inequalities for the envelope around $A$, we get that in a small neighbourhood around $A$ the inequalities of type $(**)$ are always satisfied and as in Proposition \ref{3 basic properties out-envelopes} (iii) it has full dimension $3n-4$.

On the other hand let $D \in \EnvR(A,B)$ close to $B$, i.e. with the same topological type as $B$. By lemma \ref{keep witness} also $xy$ has to be a witness from $D$ to $B$ and hence $x$ and $y$ are also a witnesses from $D$ to $B$ (for $a,b,c,d \in \R_{>0}$ we have $\frac{a+b}{c+d}\leq \frac{a}{c}$ or $\frac{b}{d}$). But this yields a non trivial equality condition in $(**)$ for the edge lengths of $D$ and so $\EnvR(A,B)$ has close to $B$ at most dimension $3n-4-1$.
\end{Bew}

\begin{Bem}
The argumentation of lemma \ref{4 geodesic_passing_face_looses_general} does not necessarily hold in the non-reduced Outer Space, e.g. consider $C$ as a doubled barbell graph.

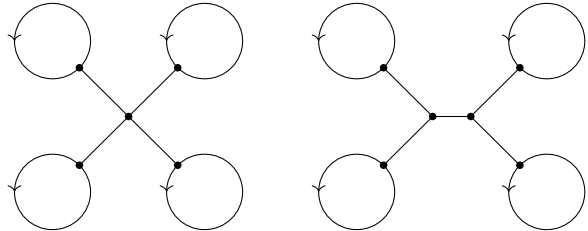
\begin{figure}[h]
\center
\begin{tikzpicture}[scale=0.5]
\draw[decoration={markings, mark=at position 0.5 with {\arrow{>}}},postaction={decorate}] (-2,-2) circle (1);
\draw[decoration={markings, mark=at position 0.5 with {\arrow{>}}},postaction={decorate}] (2,-2) circle (1);
\draw[decoration={markings, mark=at position 0.5 with {\arrow{>}}},postaction={decorate}] (-2,2) circle (1);
\draw[decoration={markings, mark=at position 0.5 with {\arrow{>}}},postaction={decorate}] (2,2) circle (1);
\draw (-1.29289,-1.29289)node[circle,fill,inner sep=1pt]{} -- (0,0) node[circle,fill,inner sep=1pt]{};
\draw (-1.29289,1.29289)node[circle,fill,inner sep=1pt]{} -- (0,0);
\draw (1.29289,-1.29289)node[circle,fill,inner sep=1pt]{} -- (0,0);
\draw (1.29289,1.29289)node[circle,fill,inner sep=1pt]{} -- (0,0);

\begin{scope}[shift={(8,0)}]
\draw[decoration={markings, mark=at position 0.5 with {\arrow{>}}},postaction={decorate}] (-2,-2) circle (1);
\draw[decoration={markings, mark=at position 0.5 with {\arrow{>}}},postaction={decorate}] (3,-2) circle (1);
\draw[decoration={markings, mark=at position 0.5 with {\arrow{>}}},postaction={decorate}] (-2,2) circle (1);
\draw[decoration={markings, mark=at position 0.5 with {\arrow{>}}},postaction={decorate}] (3,2) circle (1);
\draw (-1.29289,-1.29289)node[circle,fill,inner sep=1pt]{} -- (0,0) node[circle,fill,inner sep=1pt]{} -- (1,0) node[circle,fill,inner sep=1pt]{};
\draw (-1.29289,1.29289)node[circle,fill,inner sep=1pt]{} -- (0,0);
\draw (2.29289,-1.29289)node[circle,fill,inner sep=1pt]{} -- (1,0);
\draw (2.29289,1.29289)node[circle,fill,inner sep=1pt]{} -- (1,0);
\end{scope}

\end{tikzpicture}

\caption{doubled barbell graph and its adjacent topological type}
\end{figure}

However we resolve the 4-valent vertex, we will always get the same sets of candidates. Fixing a weighting and hence $C$ we can now easily find a neighbourhood of $C$ which contains only the topological types gained by slightly stretching the 4-valent vertex. 
\end{Bem}

With previous lemma \ref{4 geodesic_passing_face_looses_general} and the fact that envelopes are preserved under isometries we get that isometries of $CV_n^{red}$ send maximal simplices to maximal simplices. The lower dimensional skeleton of $CV_n^{red}$ is preserved for topological reasons and we get.

\begin{Satz}
An isometry in regard of the asymmetric metric of $CV_n^{red}$ is simplicial.
\end{Satz}

For the symmetric version of this theorem, which also implies last theorem, we will need a little bit more work. We will first show, that almost all points $A,B \in \Delta$ in the same maximal simplex have a top dimensional envelope $\Env(A,B)$. For this we will introduce a notion of general position. On the other hand by lemma \ref{4 geodesic_passing_face_looses_general} there are open sets at faces, where this is not satisfied.

\begin{defi}\label{def:generalPosition}
Let $A,B \in CV_n$ be in maximal simplices. We say $B$ is in \emph{general position} to $A$, if there exist neighbourhoods $U_A \ni A, U_B \ni B$ s.t. $\forall \ A' \in U_A, B' \in U_B$ the sets of candidate witnesses $CW(A,B)=CW(A',B')$ are the same. Else we say $B$ is in \emph{special position} to $A$.
\end{defi}

As one expects almost every pair is in general position, more exactly:

\begin{Prop}\label{general_position_dense}
Let $A \in CV_n$ be in a maximal simplex, then the set of points $B$ in general position to $A$ is dense and open. The same statement holds, if we fix $B$ and vary $A$.
\end{Prop}

\begin{Bew}
The open property follows directly from the definition.
The dense property follows from Lemma \ref{4_general_position_and_maximal_dimension}, since in each simplex there are only finitely many envelopes 
\end{Bew}

We can directly see points in general position with the help of in- or out-envelopes, namely
\begin{Lem}\label{4_general_position_and_maximal_dimension}
Let $A,B \in CV_n$ be points in maximal simplices, then the following are equivalent:
\begin{enumerate}[(i)]
\item $B$ is in general position to $A$.
\item $A \in  \Env_R^{in}(B,\gamma)^O$ for a $\gamma \in F_n$.
\item $B \in  \Env_R^{out}(A,\gamma)^O$ for a $\gamma \in F_n$.
\end{enumerate}
\end{Lem}

\begin{Bew}
\glqq (i)$ \Rightarrow $(ii) and (iii)\grqq \ follows directly from the definition of general position for every $\gamma \in CW(A,B)$. 

\glqq (ii)$ \Rightarrow $(i)\grqq: By proposition \ref{3 basic properties out-envelopes} (iv) wlog. $\gamma \in CW(A,B)$. Let $\omega \neq \gamma$ with $\omega \in CW(A,B)$. Since $A$ is in the interior of $\Env_R^{in}(B, \gamma)$ we have by proposition \ref{3 basic properties in-envelopes} (v) $\Env_R^{in}(B, \omega)\cap T(B)=\Env_R^{in}(B, \gamma) \cap T(B)$. By the $(\star \star)$-equality for $\gamma$ and $\omega$ then follows that the edge-counts of $\gamma$ and $\omega$ in $A$ are multiples of each and thus $\gamma$ and $\omega$ are a barbel and its counterpart with a flipped orientation of a cycle (e.g. $xy$ and $xy^{-1}$). 

We will now show, that $\gamma$ and $\beta$ also have the the same edge-counts in $B$. Let $h: A \to B$ be the mark changing map realising the minimal Lipschitz constant (s. definition \ref{defi:Lipschitz-distance}). Since $\gamma$ and $\omega$ are witnesses, their images in $B$ are reduced paths. But since they have the same edge counts in $A$ this implies they have also the same edge counts in $B$. In particular they have also the same lengths in the simplex, i.e. $l_{A'}(\omega)=l_{A}(\gamma)$ and $l_{B'}(\omega)=l_{B}(\gamma)$ for all $A' \in T(A), B' \in T(B)$.

Therefore we have $\omega \in CW(A',B')$ for all $A' \in U_A$ and $B' \in U_B$ for $U_A  \subset T(A)$ and $U_B \subset T(B)$ small enough neighbourhoods of $A$ resp. $B$. Hence we showed $CW(A,B) \subseteq CW(A',B')$ and by lemma \ref{lem:looseCWinNeighbourhood} equality holds for small enough $U_A$ and $U_B$.

\glqq (iii)$ \Rightarrow $(i)\grqq: As before if $\omega \neq \gamma$ with $\omega \in CW(A,B)$ exists, it has a multiple of edge-counts as $\gamma$ in $B$. In particular the letter count of the word $\gamma*\omega \in F_n$ has parity two, which implies they can not be extended to a basis of $F_n$. But it is easy to check that each pair of disjoint candidates which are not a barbell (or a figure of eight) and its counterpart can be extended to a basis, as the following sketches:
\begin{itemize}
\item Let $\gamma, \omega \in F_n$ be such a pair and fix a spanning tree which contains as much as possible of $\gamma$, that means all up to one or two edges.
\item Label the edges according to the marking. The labels form a basis of the fundamental group and we can exchange one of these elements with $\gamma$ (if $\gamma$ is a simple closed loop, it is already one of the labels) and still have a basis.
\item If $\omega$ is a barbell consider it as a word $w_1 w w_2 w^{-1}$ in this basis where the subwords $w_1, w, w_2$ have a disjoint alphabet, hence we can again exchange a letter of $w_1$ or $w_2$ which is not contained in $\gamma$ (since $\omega$ is not its counterpart) and again get a basis. When $\omega$ is not a barbell, its word has either again a letter not contained in $\gamma$ or it is one of the two letters of $\gamma$.
\end{itemize}
Therefore $\gamma$ and $\omega$ must be as in the case (ii) and the claim follows.
\end{Bew}

Looking at the proof of proposition \ref{3 basic properties out-envelopes} (iii) we get that $\Env_R^{out}(A,\gamma)\cap T(A)$ has maximal dimension for all $\gamma \in \cand(A)$ near $A$. Applying this and lemma \ref{4_general_position_and_maximal_dimension} to $\EnvR(A,B)=\Env_R^{out}(A,\gamma)\cap \Env_R^{in}(B,\gamma)$ for some $\gamma \in CW(A,B)$ we get the following corollary:
 
\begin{Kor}\label{4_gen_pos_max_dimension}
Let $A,B \in CV_n$ be in maximal simplices.
\begin{enumerate}[a)]
\item We have $B$ is in general position to $A$ if and only if $\EnvR(A,B)$ has maximal dimension in $T(A)$.
\item In particular if $A,B$ are in the same maximal simplex and in general position to each other, we have $\dim(\Env(A,B))=3n-4$: Since $T(A)=T(B)$ we have
\[\Env(A,B)\cap T(A)=\Env_R(A,B)\cap \Env_R(B,A)\cap T(A).\]
The straight line $[A,B]$ lies in the interior of all the considered envelopes above. So in contrast to lemma \ref{4 geodesic_passing_face_looses_general} we have $\dim(\Env(A,B) \cap U)=3n-4$ for any neighbourhood $U$ of $B$.
\end{enumerate}
\end{Kor}

We use now corollary \ref{4_gen_pos_max_dimension} b) and lemma \ref{4 geodesic_passing_face_looses_general} to distinguish faces with envelopes in the symmetric metric. Since envelopes are preserved under isometries we then get the following theorem:

\begin{Satz}\label{4 distinguish faces in CVn}
An isometry in regard of the symmetric Lipschitz-metric of $CV_n^{red}$ is simplicial.
\end{Satz}

\begin{Bew}
Let $\varphi \in \Isom(CV_n^{red})$ and $C \in CV_n$ be a point in a maximal simplex. Assume $\varphi(C)$ lies on a face and let $U \subset T(C)$ be a neighbourhood of $C$ contained in the maximal simplex.

Let $\varphi(A), \varphi(B) \in \varphi(U)$ be as in the proof of lemma \ref{4 geodesic_passing_face_looses_general}. In particular $\varphi(A)$ and $\varphi(B)$ are in maximal simplices and $CW(\varphi(A),\varphi(B))=\{xy\}$ with $x,y \in F_n$ as in the proof of lemma \ref{4 geodesic_passing_face_looses_general}. By lemma \ref{lem:looseCWinNeighbourhood} there exist neighbourhoods $\varphi(U_A), \varphi(U_B) \subseteq \varphi(U)$ such that for all $\varphi(A') \in \varphi(U_A), \varphi(B') \in \varphi(U_B)$ we have $CW(\varphi(A'),\varphi(B'))=\{xy\}$ and hence $\varphi(B')$ is in general position to $\varphi(A')$. By lemma \ref{4 geodesic_passing_face_looses_general} we can choose $U_B$ small enough such that $\Env(\varphi(A'),\varphi(B')\cap U_B)$ has not maximal dimension.

On the other hand since being in general position is a dense property (s. proposition \ref{general_position_dense}) we find $A' \in U_A$ and $B' \in U_B$ which are in general position to each other. By corollary \ref{4_gen_pos_max_dimension} b) $\Env(A',B')$ has near $B'$ maximal dimension. But $\varphi$ is an isometry and hence restricts to an isometry of $\Env(A',B')$ to $\Env(\phi(A'),\phi(B'))$. In particular it preserves the dimension of envelopes, hence the fact that $\varphi(C)$ lies on a face leads to a contradiction.

We have thus prooved that $\varphi$ sends maximal simplices to maximal simplices and thus maps the $3n-5$-skeleton of $CV_n$ to the $3n-5$-skeleton.
If $C$ belongs to the $3n-6$-skeleton of $CV_n$, then $C$ has either two $4$-valent vertices or one $5$-valent vertex which can be resolved in at least 4 different ways while staying in the reduced case. This means $C$ belongs to a face of more than two simplices of dimension $3n-5$. So by topological reasons $\varphi(C)$ must belong to the $3n-6$-skeleton of $CV_n$. Inductively $\varphi$ sends each simplex to a simplex of the same dimension.
\end{Bew}

We can now follow the proof of Stefano Francaviglias and Armando Martinos in \cite{FM12} for the non-reduced Outer Space and get the same result for reduced Outer Space, namely:

\begin{restatable}{theorem}{isomRedCV}
\label{THEOREM2}
The isometry groups of $CV_n^{red}$ in regard of the symmetric and both asymmetric Lipschitz-metrics are the same as in the non-reduced case: \[\Isom(CV_n^{red})=\Isom(CV_n)=\begin{cases}
\Out(F_n)&, \text{ if } n \geq 3\\
\PGL(2,\Z)&, \text{ if } n=2
\end{cases}\]
\end{restatable}

\begin{Bem}
Another way to distinguish faces is directly by the property of general position. By corollary \ref{4_gen_pos_max_dimension} the property of general position of two points is preserved under isometries if both points are sent into maximal simplices. If $g: [0,1] \to CV_n$ is a straight line in a maximal simplex, then by corollary \ref{4_general_position_and_maximal_dimension} b) $g(1)$ is in general position to $g(0)$ if and only if $g(t)$ is in general position to $g(s)$ for all $0 \leq s < t \leq 1$. On the other hand let $A$ and $B$ as in lemma \ref{4 geodesic_passing_face_looses_general} and $g: [0,1] \to CV_n$ any geodesic from a $A$ to $B$. Then there always exists $0<s<t<1$, such that $g(s)$ and $g(t)$ are not in general position, since as soon as $g$ passes the face it has to lie in an envelope of smaller dimension. This behaviour just relies on the fact, that the coarse direction of $g$ is not a candidate of (the simplex of) $B$.
\end{Bem}

Having this in mind we can actually deduce similar behaviour for all geodesic rays, namely that if a geodesic runs long enough it looses one dimension of freedom or in other words gains a small amount of rigidity:

\begin{Prop}\label{4_long_geodesics_not_general}
Let $g: \R_{\geq 0} \to CV_n$ be an asymmetric geodesic ray parametrized by length, then there exists a $T \in \R$ s.t. for all $T<s<t$ we have $\dim(\Env_R(g(s),g(t)) \leq \dim(CV_n)-1$. In particular $g(t)$ is not in general position to $g(s)$.
\end{Prop}

\begin{Bew}
Let $\gamma \in \cap_{t \in \R} CW(g(0),g(t))$ a coarse direction of $g$. Since $g$ is a (global) geodesic ray, such a $\gamma$ exists and by lemma \ref{keep witness} we have $\gamma \in W(g(s),g(t))$ for all $0<s<t$. Since $\gamma$ is a witness for all elements of the path, its length is stretched exponentionally to the length path, hence there exists a $T>0$ with $l_{g(s)}(\gamma) > 2 \vol(g(s))$ for all $s>T$. In particular $\gamma$ is not a candidate for all $g(s)$ with $s>T$ since at least one edge is covered by $\gamma$ thrice.\\
Let now $\omega \in CW(g(s),g(t))$ for some $t>s>T$. Then the equalities in $(\star)$ for $A=g(s)$ and $\omega$ and $\gamma$ yield a non-trivial restriction for $g(t)$, else we would have $(l_{g(s)}(\omega) \cdot \#(e_i, \gamma)- l_{g(s)}(\gamma) \cdot \#(e_i, \omega))=0$ for all edges $e_i$, in other words in the abelianization of $F_n$ $\gamma$ would be a $l_{g(s)}(\gamma)/l_{g(s)}(\omega)$ multiple of $\omega$. But at least one edge is covered multiple times by $\gamma$ and thus $l_{g(s)}(\gamma)>l_{g(s)}(\omega)$, hence $\gamma$ would be a proper multiple in the abelianization. But $\gamma$ can be extended to a basis of $F_n$ an thus to a basis of the abelianization.\\
Since the geodesic ray has to satisfy this nontrivial equality for $t>T$ we have \linebreak
$\dim(\Env_R^{out}(g(s))) < 3n-4$ and by corollary \ref{4_gen_pos_max_dimension} $g(t)$ is not in general position to $g(s)$.
\end{Bew}

Since a geodesic is rigid if and only if the envelope has everywhere dimension 1, we get the following corollary:

\begin{Kor}
Let $g: \R_{\geq 0} \to CV_2$ be an asymmetric geodesic ray parametrized by length, then there exists a $T\geq 0$ such that $g|_{\R_{\geq T}}$ is a rigid geodesic ray.
\end{Kor}

On the other hand proposition \ref{4_long_geodesics_not_general} is sharp for every $n$, i.e. there exists long geodesic rays $g: \R_{\geq 0} \to CV_n$ with $\dim(\Env_R(g(s),g(t)))=\dim(CV_n)-1$ for all $0<s<t$.

\begin{Bsp}
Let $A_0 \in CV_2$ be the figure of eight 
\begin{figure}[h]
\center
\begin{tikzpicture}
\draw[red,postaction={decorate,decoration={
		markings, mark=at position .5 with {\arrow{stealth}} }}] (0,0) arc (180:-180:1/2);
\node[blue] at (-1.2,-0.5) {$a$};
\draw[blue,postaction={decorate,decoration={
		markings, mark=at position .5 with {\arrow{stealth}} }}] (0,0) arc (0:360:1/2);
\draw[fill] circle [radius=0.05];
\node[red] at (1.2,-0.5) {$1-a$};
\end{tikzpicture}
\caption{marked figure of eight}
\end{figure}
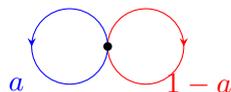
with marking $\textcolor{blue}{x}$ and $\textcolor{red}{y}$ and edge lengths $\textcolor{blue}{a}$ and $\textcolor{red}{1-a}$ for $a=\frac{\sqrt{5}-1}{2}$.
If we choose $g$ to be the geodesic ray starting at $A_0$ with direction $\{xy\}$ which is the same direction as $\{x,y\}$ a short calculation shows, that $g$ is the unique geodesic from $A_0$ to $A_1$ where $A_1$ is the figure of eight with marking $\textcolor{red}{xy^{-1}}, \textcolor{blue}{y}$ and again with edge-lengths $\textcolor{red}{1-a}$ and $\textcolor{blue}{a}$.
In terms of envelopes the simplex $\Delta_0$ containing $A_0$ and $A_1$ is covered by the envelopes of $\Env_R^{out}(A_0,x)$ and $\Env_R^{out}(A_0,y)$ and $g$ is exactly the intersections of those two envelopes and hence rigid. 

Furthermore for the adjacent $\theta$-simplex $\Delta_1$  containing $A_1$ the out-envelopes \linebreak
$\Env_R^{out}(A_1,x)$ and $\Env_R^{out}(A_1,y)$ lay in the interior of $\Delta_1$, hence the envelope \linebreak
$\EnvR^{out}(A_1, \{x,y\}) \cap \Delta_1$ is one dimensional and by proposition \ref{3 basic properties out-envelopes} (iv) it is the intersection of out-envelopes of candidates of $A_1$. But these look exactly like for $A_0$ and $\Delta_0$, i.e. there are only two envelopes belonging to $xy^{-1}$ and $y$ and hence $\EnvR^{out}(A_1, \{x,y\}) \cap \Delta_1$ intersects the next face at $A_2$ with again edge lengths $a$ and $1-a$.

Since we can continue $g$ along the out-envelopes with coarse direction $\{x,y\}$  we inductively get that $g$ is a rigid geodesic ray with infinite length. Figure \ref{fig: infinite long geodesic} is a picture of $g$ where the letters in the brackets denote the marking of the figures of eights.

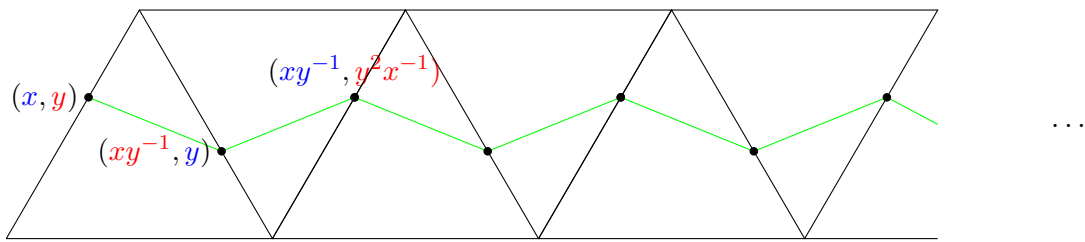
\begin{figure}[h]\label{fig: infinite long geodesic}
\center
\begin{tikzpicture}[scale=3.5]
\draw[green] (0.309017, 0.535233)--(0.809017, 0.33079)--(1.309017,0.535233);
\draw (0,0)--(1,0) -- (0.5, 0.8660254) -- (0,0);
\draw (0.5, 0.8660254) --  (1.5, 0.8660254) -- (1,0);
\filldraw (0.309017, 0.535233) circle[radius=0.4pt] node[left]{$(\textcolor{blue}{x},\textcolor{red}{y})$};
\filldraw (0.809017, 0.33079) circle[radius=0.4pt] node[left]{$(\textcolor{red}{xy^{-1}},\textcolor{blue}{y})$};
\filldraw (1.309017, 0.535233) circle[radius=0.4pt] node[above]{$(\textcolor{blue}{xy^{-1}},\textcolor{red}{y^2x^{-1})}$};

\begin{scope}[shift={(1,0)}]
\draw[green] (0.309017, 0.535233)--(0.809017, 0.33079)--(1.309017,0.535233);
\draw (0,0)--(1,0) -- (0.5, 0.8660254) -- (0,0);
\draw (0.5, 0.8660254) --  (1.5, 0.8660254) -- (1,0);
\filldraw (0.309017, 0.535233) circle[radius=0.4pt];
\filldraw (0.809017, 0.33079) circle[radius=0.4pt];
\filldraw (1.309017, 0.535233) circle[radius=0.4pt];

\end{scope}
\begin{scope}[shift={(2,0)}]
\draw[green] (0.309017, 0.535233)--(0.809017, 0.33079)--(1.309017,0.535233)--(1.5, 0.433);
\draw (0,0)--(1,0) -- (0.5, 0.8660254) -- (0,0);
\draw (0.5, 0.8660254) --  (1.5, 0.8660254) -- (1,0);
\filldraw (0.309017, 0.535233) circle[radius=0.4pt];
\filldraw (0.809017, 0.33079) circle[radius=0.4pt];
\filldraw (1.309017, 0.535233) circle[radius=0.4pt];

\end{scope}
\draw (3,0)--(3.5,0);
\node at (4, 0.43) {\dots};
\end{tikzpicture}

\caption{An infinite long rigid geodesic ray in $CV_2$} 
\end{figure}

If we now consider the graphs $A_i'$ where we add to each $A_i$ the same additional marked subgraph, e.g. a bouqet of roses with petal length 1 and the same marking.
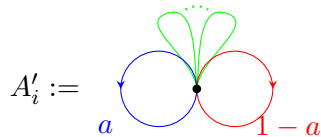
\[
\begin{tikzpicture}
\node at (-2,0) {$A_i':=$};
\draw[green] (0,0) to [out=95,in=200]  (-0.4,1) to [out=20,in=95] (0,0);
\draw[green] (0,0) to [out=85,in=-20]  (0.4,1) to [out=160,in=85] (0,0);
\draw[green, thick, dotted] (-0.15,1.051) to [out=30,in=150]  (0.15,1.051);
\draw[red,postaction={decorate,decoration={
		markings, mark=at position .5 with {\arrow{stealth}} }}] (0,0) arc (180:-180:1/2);
\node[blue] at (-1.2,-0.5) {$a$};
\draw[blue,postaction={decorate,decoration={
		markings, mark=at position .5 with {\arrow{stealth}} }}] (0,0) arc (0:360:1/2);
\draw[fill] circle [radius=0.05];
\node[red] at (1.2,-0.5) {$1-a$};
\end{tikzpicture}\]
it is easy to check, that the added green part doesn't distribute anything to the distance to the other $A_j'$. Hence we can slightly vary the corresponding green edges without changing the distance. In particular we get that $g'$ as the concatenation of the straight edges between the $A_i'$ is an infinite geodesic ray with $\EnvR(g(s),g(t))=\dim(CV_n)-1$ for all $0\leq s < t$ since the green part is never covered by a witness from $g(s)$ to $g(t)$.
\end{Bsp}

\section{Appendix}
\begin{Bem}[Calculations for \ref{non_geodesic_ball_dim2}]\label{bem:app_calc_non_geodesic}
Let $X$ be the figure of eight graph with marking $\textcolor{blue}{x}$ and $\textcolor{red}{y}$ and edge length $a$ and $1-a$ as in figure \ref{fig:egged_graphs_rem5.1}. Let furthermore $A, C$ be as in figure \ref{fig:egged_graphs_rem5.1} the two differently marked theta-graphs in the neighbourhood of $X$ with edge lengths $a^\pm = a \pm \delta, \varepsilon$ and $1-a^\pm-\varepsilon$ for some small enough $\delta,\varepsilon>0$.

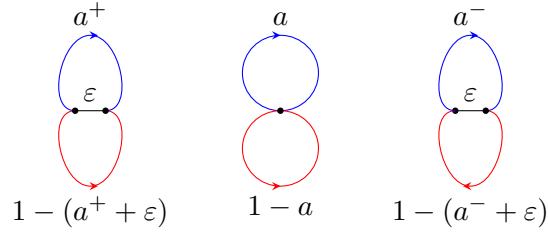
\begin{figure}[h]\label{fig:egged_graphs_rem5.1}
\center
\begin{tikzpicture}[decoration={
		markings, mark=at position .5 with {\arrow{stealth}} }]

\draw (-0.2,0) --node[above]{$\varepsilon$} (0.2,0);
\draw[red] (-0.2,0) to[out=180, in=180] (0,-1) to[out=0, in=0] (0.2,0);
\draw[red,postaction={decorate,decoration}](0.099,-1)--(0.1,-1);
\node[above] at (0,1) {$a^+$};
\draw[blue] (-0.2,0) to[out=180, in=180] (0,1) to[out=0, in=0] (0.2,0);
\draw[blue,postaction={decorate,decoration}](0.099,1)--(0.1,1);
\node[below] at (0,-1) {$1-(a^++\varepsilon)$};
\filldraw (-0.2,0) circle[radius=1pt];
\filldraw (0.2,0) circle[radius=1pt];

\begin{scope}[shift={(2.5,0)}]
\draw[blue, postaction={decorate,decoration}] (0,0) arc (-90:-450:1/2);
\node[above] at (0,1) {$a$};
\draw[red, postaction={decorate,decoration}] (0,0) arc (90:450:1/2);
\node[below] at (0,-1) {$1-a$};
\filldraw (0,0) circle[radius=1pt];

\begin{scope}[shift={(2.5,0)}]
\filldraw (-0.2,0) circle[radius=1pt];
\draw (-0.2,0) --node[above]{$\varepsilon$} (0.2,0);
\draw[red] (-0.2,0) to[out=180, in=180] (0,-1) to[out=0, in=0] (0.2,0);
\draw[red,postaction={decorate,decoration}](-0.099,-1)--(-0.1,-1);
\node[above] at (0,1) {$a^-$};
\draw[blue] (-0.2,0) to[out=180, in=180] (0,1) to[out=0, in=0] (0.2,0);
\draw[blue,postaction={decorate,decoration}](0.099,1)--(0.1,1);
\node[below] at (0,-1) {$1-(a^-+\varepsilon)$};
\filldraw (0.2,0) circle[radius=1pt];

\end{scope}
\end{scope}
\end{tikzpicture}

\caption{$A,X$ and $C$ of example \ref{non_geodesic_ball_dim2}}
\end{figure}

We only consider the candidates $x,y,xy$ and $xy^-$ for maximal stretching/shrinking from $A$ to $C$.
The ratios of the paths are as follow:
\begin{center}
\scalebox{1.2}{
\renewcommand{\arraystretch}{1.5}
  \begin{tabular}{ c || c | c | c | c }
	& $x$ & $y$ & $xy$ & $xy^{-1}$\\ \hline \hline    
    $\frac{A}{C}$ & $\frac{a^+ + \varepsilon}{a^-+\varepsilon}$&$\frac{1-a^+}{1-a^-}$&$\frac{1+\varepsilon}{1-\varepsilon}$& $\frac{1-\varepsilon}{1+\varepsilon}$\\ \hline
    $\frac{A}{B}$ & $\frac{a^+ + \varepsilon}{\alpha}$&$\frac{1-a^+}{1-\alpha}$&$1+\varepsilon$&$1-\varepsilon$\\ \hline
    $\frac{B}{C}$ & $\frac{\alpha}{a^-+\varepsilon}$&$\frac{1-\alpha}{1-a^-}$&$\frac{1}{1-\varepsilon}$&$\frac{1}{1+\varepsilon}$\\
  \end{tabular}
  }
\end{center}
Where $B$ is some graph in the same simplex as $X$ with edge length $\alpha$ and $1-\alpha$.\\
It is clear, that the ratios of lengths satisfy $\frac{A}{C}(x)>1>\frac{A}{C}(y)$ and  $\frac{A}{C}(xy)>\frac{A}{C}(xy^{-1})$. For $\delta<a \varepsilon$ we get the following witnesses:
\begin{itemize}
\item $xy$ is maximally shrinked from $A$ to $C$:\\
It is enough to show  $\frac{A}{C}(xy)>\frac{A}{C}(x)$:
\begin{align*}
\frac{1+\varepsilon}{1-\varepsilon}>\frac{a^+ + \varepsilon}{a^-+\varepsilon} &\iff (1+ \varepsilon)(a^-+\varepsilon)>(1- \varepsilon)(a^++\varepsilon)\\
&\iff 2(-\delta+\varepsilon(a+\varepsilon))>0 \iff \delta<a\varepsilon+\varepsilon^2
\end{align*}
which is satisfied since we have $\delta<a\varepsilon$.
\item $xy^-1$ is maximally stretched from $A$ to $C$:
\begin{align*}
\frac{1-\varepsilon}{1+\varepsilon}<\frac{1-a^+}{1-a^-} &\iff (1- \varepsilon)(1-a^-)<(1+\varepsilon)(1-a^+)\\
&\iff 2(\delta-\varepsilon(1-a))<0\iff \delta < (1-a)\varepsilon
\end{align*}
which is satisfied, since we have $\delta<a\varepsilon\leq (1-a)\varepsilon$.
\end{itemize}
Thies yields as conditions for any $B$ on a geodesic from $A$ to $C$:
\begin{itemize}
\item $xy$ is maximally stretched from $B$ to $A$:
\begin{align*}
\frac{A}{B}(xy)\geq\frac{A}{B}(x) &\iff \alpha\geq\frac{a^+ +\varepsilon}{1+\varepsilon}\\
\frac{A}{B}(xy)\geq\frac{A}{B}(y) &\iff \frac{1-a^+}{1+\varepsilon}\leq 1-\alpha &\\
\iff \alpha \leq \frac{1+\varepsilon-(1-a^+)}{1+\varepsilon}=\frac{a^+ +\varepsilon}{1+\varepsilon}
\end{align*}
\item $xy^{-1}$ is maximally stretched from $A$ to $B$:
\begin{align*}
\frac{A}{B}(xy^{-1})\leq\frac{A}{B}(y) &\iff 1-\alpha \geq \frac{1-a^-}{1+\varepsilon}\\
&\iff \alpha \leq \frac{1+\varepsilon-1+a^-}{1+\varepsilon} = \frac{a^- + \varepsilon}{1+ \varepsilon}<\frac{a^+ +\varepsilon}{1+\varepsilon}=\alpha
\end{align*}
hence we have a contradiction and such an $\alpha$ can not exist.
\end{itemize}
\end{Bem}

\begin{Bem}[geodesic joining adjacent barbell- and $\theta$-graphs]\label{bem:app_geod_join_barb_theta}
Let $A$ and $C$ be elements in $CV_n$ with marking $\textcolor{blue}{x}$ and $\textcolor{red}{y}$.
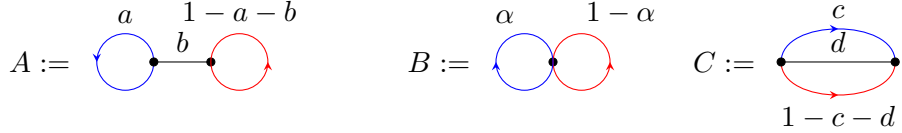
\begin{center}
\begin{tikzpicture}[scale=1.5, decoration={
		markings, mark=at position .5 with {\arrow{stealth}} }]
		\node at (-1.5,0) {$A:=$};
\draw[postaction={decorate,decoration}, blue] (-0.5,0) arc (0:360:1/4);
\node[above] at (-0.75,0.25) {$a$};
\draw (-0.5,0) -- node[above]{$b$} (0,0);
\filldraw (-0.5,0) circle[radius=1pt];
\filldraw (0,0) circle[radius=1pt];
\draw[postaction={decorate,decoration}, red] (0,0) arc (-180:180:1/4);
\node[above] at (0.25,0.25) {$1-a-b$};

\begin{scope}[shift={(3,0)}]
\node at (-1,0) {$B:=$};
\filldraw (0,0) circle[radius=1pt];
\node[above left] at (-0.25,0.25) {$\alpha$};
\node[above right] at (0.2,0.25) {$1-\alpha$};
\draw[red, postaction={decorate,decoration}] (0,0) arc (-180:180:1/4);
\draw[blue, postaction={decorate,decoration}] (0,0) arc (0:-360:1/4);
\end{scope}
\begin{scope}[shift={(5,0)}]
\node at (-0.5,0) {$C:=$};
\draw(0,0) to node[above]{$d$} (1,0);
\draw[postaction={decorate,decoration}, blue] (0,0) to [out=90,in=90] (1,0);
\node[above] at (0.5, 0.3) {$c$} ;
\draw[postaction={decorate,decoration}, red] (0,0) to [out=270,in=270] (1,0);
\node[below] at (0.5, -.3){$1-c-d$} ;
\filldraw (0,0) circle[radius=1pt];
\filldraw (1,0) circle[radius=1pt];
\end{scope}
\end{tikzpicture}
\end{center}
then there exists a symmetric geodesic from $A$ to $C$.

We will show you can choose $\alpha$ in such a manner, that it lies on a geodesic joining $A$ and $C$. By lemma \ref{glueing geodesics} it is enough to consider the witnesses from $A$ to $B$ and from $B$ to $C$, since $A$ and $B$ lie in the same simplex as do $B$ and $C$.\\
As above we only consider the candidates $x,y,xy$ and $xy^-$ for maximal stretching/shrinking from $A$ to $C$.
The ratios of the paths are as follow:
\begin{center}
\scalebox{1.2}{
\renewcommand{\arraystretch}{1.5}
  \begin{tabular}{ c || c | c | c | c }
	& $x$ & $y$ & $xy$ & $xy^{-1}$\\ \hline \hline    
    $\frac{A}{C}$ & $\frac{a}{c+d}$&$\frac{1-a-b}{1-c}$&$\frac{1+b}{1-d}$& $\frac{1+b}{1+d}$\\ \hline
 $\frac{A}{B}$ & $\frac{a}{\alpha}$&$\frac{1-a-b}{1-\alpha}$&$1+b$& $1+b$\\ \hline
 $\frac{B}{C}$ & $\frac{\alpha}{c+d}$&$\frac{1-\alpha}{1-c}$&$\frac{1}{1-d}$& $\frac{1}{1+d}$\\ 
  \end{tabular}
 }
\end{center}
Wlog. we can assume $\frac{A}{C}(x)\leq \frac{A}{C}(y)$ hence we get
\begin{align*}
\frac{A}{C} (x)\leq \frac{A}{C}(y)&\iff \frac{a}{c+d} \leq \frac{1-a-b}{1-c}\\
&\iff (1-c+c+d)a \leq (c+d)(1-b) \iff a \leq \frac{(c+d)(1-b)}{1+d}\\
&\iff \frac{a}{c+d}\leq \frac{1-b}{1+d}\leq \frac{1+b}{1+d}
\end{align*}
So we get that the ratio of lengths satisfy $\frac{A}{C}(xy)>\frac{A}{C}(xy^-)> \frac{A}{C}(x)$ and hence the maximally stretched path from $A$ to $C$ is $x$.
This yields the following restrictions for $\alpha$:
\begin{itemize}
\item $\frac{A}{B} (x)\leq \frac{A}{B} (y)$: $a(1-\alpha)\leq \alpha (1-a-b) \iff \alpha \geq \frac{a}{1-a+a-b}=\frac{a}{1-b}$
\item $\frac{A}{B} (x)\leq \frac{A}{B} (xy^-)$: $\alpha \geq \frac{a}{1+b}$
\item $\frac{B}{C} (x)\leq \frac{B}{C} (y)$: $(1-c)\alpha \leq (c+d)(1-\alpha) \iff \alpha \leq \frac{c+d}{1-c+c+d}=\frac{c+d}{1+d}$
\item $\frac{B}{C} (x)\leq \frac{B}{C} (xy^-) \leq \frac{B}{C}(xy)$: $\alpha \leq \frac{c+d}{1+d}$
\end{itemize}
Hence we have the restriction $\frac{a}{1-b} \leq \alpha \leq \frac{c+d}{1+d}$

Let's consider the maximally shrinked paths:\\
Consider the case $y$ is maximally shrinked (by above we can ignore $x$ and $xy^-\leq xy$ is always true):
\begin{itemize}
\item $\frac{A}{B} (xy)\leq \frac{A}{B} (y)$: $1- \alpha \leq \frac{1-a-b}{1+b} \iff \alpha \geq \frac{1+b-1+a+b}{1+b}=\frac{a+2b}{1+b}$ ($\geq \frac{a}{1-b}$ since $a < 1-b$)
\item $\frac{B}{C} (xy)\leq \frac{B}{C} (y): 1-\alpha \geq \frac{1-c}{1-d} \iff \alpha \leq \frac{1-d-1+c}{1-d}=\frac{c-d}{1-d} (\leq \frac{c+d}{1+d})$
\end{itemize}
So we have only these restrictions, which are fullfilled for some $\alpha$ since they are coming from a one sided geodesic.

On the other hand if $xy$ is maximally shrinked, then above inequalities only flip. But since the now higher bound from $xy$ maximal shrinked is bigger than the lower bound from $x$ maximal stretched and vice versa, the two possible intervals for $\alpha$ intersect non-empty and we have a possible $\alpha$ for the two-sided geodesic.
\end{Bem}

\begin{Bem}[proof of remark \ref{2 fun stuff with locally-minimizing geodesics} (ii)]
Let $g: I \to CV_n$ be a continuous path. As usually for any given $\varepsilon >0$ we can find a piecewise geodesic path $\tilde{g}$ which lies in a $\varepsilon$ neighbourhood of $g$ and vice versa. The problem pins now down to changing the coarse direction at the vertices of $\tilde{g}$, i.e. the points where $\tilde{g}$ is not necessarily geodesic. We can do that by iteratively jumping over the faces of the out-envelopes by consecutively going a tiny segment along each face until we reach a face of the corresponding out-envelope of our wanted direction. An example for such a face jumping path would be remark \ref{2 fun stuff with locally-minimizing geodesics} (i). The explicit construction is as follows.

\begin{tikzpicture}
\end{tikzpicture}

Let $A,B$ be two adjacent vertices of $\tilde{g}$. By proposition \ref{general_position_dense} we can wlog. assume that $B$ is in general position to $A$. Furthermore we can assume $d(A,B)<\varepsilon$.

Let $\gamma \in CW(A,B)$ and $0 < \delta$ be small enough, such that the following holds for all $A' \in B_\delta(A):=\{A' \in CV_n \ | \ d(A,A')<\delta\}$
\begin{enumerate}[(i)]
\item $\delta < \varepsilon$ and $B_\delta(A) \subset T(A)$
\item $B \in \Env_R^{out}(A',\gamma)$
\item We have for all $\omega, \varphi \in \cand(A)$ if $\Env_R^{out}(A,\omega)$ and $ \Env_R^{out}(A,\varphi)$ intersect at a hyperplane, then $\Env_R^{out}(A',\omega)$ and $ \Env_R^{out}(A',\varphi)$ intersect at a hyperplane.
\end{enumerate}
Where (ii) can be satisfied by the definition of general position and since out-envelopes from $A'$ in $T(A)$ depend continuously on $A'$ (as can be seen in the $(\star)$-inequality of lemma \ref{wiggling polytope is polytope}) we can choose $\delta$ small enough to satisfy (iii).

Let now $\gamma_0 \in \cand(A)$ be a coarse direction just before $\tilde{g}$ enters $A$ and $\gamma_1, \dots, \gamma_n=\gamma \in \cand(A)$ such that $\Env_R^{out}(A,\gamma_i)$ and $\Env_R^{out}(A,\gamma_{i+1})$ intersect at a hyperplane. For $A_0:=A$ and by the properties of $B_\delta(A)$ we can now inductively find $A_i \in B_{\delta/n}(A_{i-1}) \cap \Env_R^{out}(A_{i-1},\gamma_{i-1})\cap \Env_R^{out}(A_{i-1},\gamma_i)$ for $i \in \{1, \dots, n\}$ and set $A_{n+1}:=B$. By lemma \ref{glueing geodesics} we can now glue any geodesics $g_i$ from $A_i$ to $A_{i+1}$
to get a locally minimizing geodesic $\bar{g}|_{[A,B]}:= g_0 * \dots * g_n$ from $A$ to $B$ with $\gamma_0$ and $\gamma$ as coarse directions near $A$ resp. $B$. This means, by lemma \ref{glueing geodesics} the concatenation of all segments $\bar{g}|_{[A,B]}$ for $A, B$ vertices of $\tilde{g}$ as before is a locally minimizing geodesic.

By $\delta<\varepsilon$ we get $A_1, \dots, A_n \in B_\varepsilon(A), g_0, \dots, g_{n-1} \subset B_\varepsilon(A)$ and since $d(A,B)< \varepsilon$ also $g_n \subset B_{2\varepsilon}(B)$. Hence $\bar{g}$ is in a $2\varepsilon$-neighbourhood of $\tilde{g}$ and thus in a $3\varepsilon$-neighbourhood of $g$. On the other hand since $\tilde{g}$ is piecewise geodesic we also have $\tilde{g}$ is in a $\varepsilon$-neighbourhood of $\bar{g}$ and so $g$ is in a $2\varepsilon$-neighbourhood of $\bar{g}$.

For the dense statement take any dense continuous ray $g$, which exists since $CV_n$ is a connected simplicial complex with countably many simplices. Partition $g$ into a sequence $(g_i)_{i \in \N}$ parts and apply previous statement on each $g_i$ with $\varepsilon_i=\frac{1}{n}$. Since we can choose the $\tilde{g}_i$ to have fitting endpoints and the $\bar{g}_i$ don't change them we are done.
\end{Bem}

\bibliographystyle{amsalpha}
\bibliography{references}

\end{document}